\documentclass[11pt,reqno]{amsart}
\usepackage[utf8]{inputenc}
\usepackage[margin=1in]{geometry}

\usepackage{amsmath, amssymb, amsthm}
\usepackage{mathrsfs}
\usepackage{wasysym}
\usepackage{gensymb}
\usepackage{mathtools}
\usepackage{cancel}

\usepackage{xspace}
\usepackage{xifthen}
\usepackage{setspace}

\usepackage{hyperref}
\hypersetup{colorlinks=true, citecolor=blue, linkcolor=blue, urlcolor=blue, pdfstartview=FitH}
\usepackage{cleveref}
\usepackage{enumerate}
\usepackage{array}

\usepackage[dvipsnames]{xcolor}
\usepackage{ytableau}
\usepackage{soul}
\usepackage{mdframed}
\usepackage{tcolorbox}
\tcbuselibrary{theorems}
\mdfsetup{everyline=true}
\usepackage{framed}

\usepackage{tikz}
\usepackage{tikz-cd}
\usetikzlibrary{shapes.geometric}
\usetikzlibrary{decorations.pathmorphing}
\usepackage[all,2cell,cmtip]{xy}
\UseTwocells
\usepackage{tikz-qtree}
\usepackage{forest}

\usepackage{diagbox}

\usepackage{multicol}
\usepackage{bm}
\usepackage{bbm}

\usepackage{graphicx}
\usepackage{pdfpages}

\usepackage[colorinlistoftodos,color=red!60]{todonotes}

\setul{0.5ex}{0.3ex}
\setulcolor{Plum}

\usepackage{pifont}

\renewcommand*{\d}{\mathop{}\!d}

\newcommand{\abs}[1]{\left\lvert#1\right\rvert}
\newcommand{\de}{\delta}
\newcommand{\e}{\varepsilon}

\renewcommand*{\Re}{\operatorname{Re}}

\theoremstyle{plain} 
\newtheorem{theorem}{Theorem}[section]
\newtheorem*{theorem*}{Theorem}
\newtheorem{prop}[theorem]{Proposition}

\newtheorem{lemma}[theorem]{Lemma}
\newtheorem{corollary}[theorem]{Corollary}
\newtheorem{conjecture}[theorem]{Conjecture}

\newtheorem*{example}{Example}

\theoremstyle{remark}

\newtheorem*{remark}{Remark}

\usepackage[maxnames=10,minnames=10,style=numeric]{biblatex}
\renewbibmacro{in:}{}
\usepackage{enumitem}

\usepackage{caption}
\title{Biases among Congruence Classes for Parts in $k$-regular Partitions}

\author[F. Jackson]{Faye Jackson}
\address{University of Virginia, 141 Cabell Drive, Kerchof Hall, Charlottesville, VA 22904}
\curraddr{University of Michigan, 530 Church St, Ann Arbor, MI 48109}
\email{alephnil@umich.edu}
\urladdr{\href{http://www-personal.umich.edu/~alephnil}{http://www-personal.umich.edu/~alephnil}}

\author[M. Otgonbayar]{Misheel Otgonbayar}
\address{University of Virginia, 141 Cabell Drive, Kerchof Hall, Charlottesville, VA 22904}
\curraddr{Massachusetts Institute of Technology, 77 Massachusetts Avenue
Cambridge, MA 02139-4307}
\email{misheel@mit.edu}

\date{\today}

\makeatletter
\providecommand\@dotsep{5}
\def\listtodoname{TODOS: Not Done (Red), Cite (Orange), Blue (Notation)}
\def\listoftodos{\@starttoc{tdo}\listtodoname}
\makeatother

\newcommand{\Li}{\operatorname{Li}}

\newcommand{\lr}[1]{\left(#1\right)}

\newcommand{\linecomment}[1]{}
\renewcommand{\P}{\mathcal{P}}
\newcommand{\D}{\mathcal{D}}
\newcommand{\R}{\mathbb{R}}
\newcommand{\CC}{\mathbb{C}}
\DeclareMathOperator{\csch}{csch}

\newcommand{\Err}{\mathcal{E}}
\newcommand{\Mod}[1]{\ (\mathrm{mod}\ #1)}
\DeclarePairedDelimiter{\floor}{\lfloor}{\rfloor}
\DeclareRobustCommand{\eulerian}{\genfrac<>{0pt}{}}
\numberwithin{equation}{section}
\newcommand{\eeta}{\eta}

\crefname{equation}{equation}{equations}
\crefname{prop}{proposition}{propositions}
\Crefname{prop}{Proposition}{Propositions}

\begin{document}

\begin{abstract}
    For integers $k,t \geq 2$ and $1\leq r \leq t$, let $D_k(r,t;n)$ be the number of parts among all $k$-regular partitions (i.e., partitions of $n$ where all parts have multiplicity less than $k$) of $n$ that are congruent to $r$ modulo $t$. Using the circle method, we obtain the asymptotic
    \[
	    D_{k}(r,t;n) = \frac{3^{\frac{1}{4}}e^{\pi\sqrt{\frac{2Kn}{3}}}}{\pi t 2^{\frac{3}{4}}K^{\frac{1}{4}}n^{\frac{1}{4}}\sqrt{k}}\left(\log k + \left(\frac{3\sqrt{K}\log k}{8\sqrt{6}\pi} - \frac{t\pi(k-1)K^{\frac{1}{2}}}{2\sqrt{6}}\left(\frac{r}{t}- \frac{1}{2}\right)\right)n^{-\frac{1}{2}} + O(n^{-1})\right),
    \]
    where $K = 1 - \frac{1}{k}$. The main term of this asymptotic does not depend on $r$, and so if $P_k(n)$ is the total number of parts among all $k$-regular partitions of $n$, we have that $\frac{D_k(r,t;n)}{P_k(n)} \to \frac{1}{t}$ as $n \to \infty$. Thus, in a weak asymptotic sense, the parts are equidistributed among congruence classes. However, inspection of the lower order terms indicates a bias towards the lower congruence classes; that is, for $1\leq r < s \leq t$ we have $D_k(r,t;n) \geq D_k(s,t;n)$ for sufficiently large $n$. We make this inequality explicit, showing that for $3 \leq k \leq 10$ and $2 \leq t \leq 10$, the inequality $D_k(r,t;n) \geq D_k(s,t;n)$ holds for all $n \geq 1$ and the strict inequality $D_k(r,t;n) > D_k(s,t;n)$ holds for all $n \geq 17$. 
\end{abstract}

\maketitle


\section{Introduction}

A \textit{partition} of a positive integer $n$ is a nonincreasing sequence $\lambda = (\lambda_1,\ldots,\lambda_\ell)$ of positive integers which sum to $n$, which we denote by $\lambda \vdash n$. The $\lambda_j$ are called the \textit{parts} of the partition $\lambda$. Partitions are widely studied throughout both classical and modern mathematics, and of particular interest are asymptotic expansions and exact formulae for counting functions related to partitions. In particular, let $p(n)$ be the number of partitions of $n$. Hardy and Ramanujan \cite{hardy-ramanujan} famously discovered the following asymptotic formula for $p(n)$:
\begin{equation}
    p(n) \sim \frac{1}{4n\sqrt{3}}e^{\pi\sqrt{\frac{2n}{3}}}.
    \label{eq:hi}
\end{equation}
Their proof was revolutionary, as it birthed the circle method, a now ubiquitous tool in analytic number theory used to extract asymptotic expansions for the coefficients of generating functions which have manageable singularities on the unit circle. 

A natural question that has been asked recently is how many parts among the partitions of $n$ lie in some specified arithmetic progression. In particular, Beckwith and Mertens \cite{beckwith2017,beckwith2015} answered this question. Formally, for positive integers $1 \leq r \leq t$ and $n \geq 1$,
they defined the function\footnote{In \cite{beckwith2017,beckwith2015}, the function $T(r,t;n)$ is denoted $\widehat{T}_{r,t}(n)$}
\begin{align*}
	T(r,t;n) \coloneqq \sum_{\lambda \vdash n} \#\{\lambda_j \colon \lambda_j \equiv r \Mod t\},
\end{align*}
and they proved (see \cite[Theorem~1.3]{beckwith2017}) that
\begin{align}
	T(r, t;n) = e^{\pi\sqrt{\frac{2n}{3}}} n^{-\frac{1}{2}} \frac{1}{4\pi t\sqrt{2}}\left[\log(n) - \log\left(\frac{\pi^2}{6}\right) - 2\left(\psi\left(\frac{r}{t}\right) + \log(t)\right) + O\left(n^{-\frac{1}{2}}\log(n)\right)\right] \label{eq:beckwith-asym}
\end{align}
where $\psi(x) \coloneqq \frac{\Gamma'(x)}{\Gamma(x)}$ is the logarithmic derivative of the $\Gamma$ function. 


In a similar spirit, Craig considered partitions where all of the parts are distinct, where it is known that
\[
q(n) := \#\{\lambda \vdash n \colon \lambda \text{ has distinct parts}\} \sim \frac{3^{\frac{3}{4}}}{12n^{\frac{3}{4}}}e^{\pi\sqrt{\frac{n}{3}}}.
\]
Craig \cite{craig} proved asymptotics for the number of parts congruent to $r \Mod t$ among partitions of $n$ into distinct parts. Craig defines\footnote{In \cite{craig}, the function $D(r,t;n)$ is denoted $D_{r,t}(n)$}
\begin{align*}
	D(r,t;n) \coloneqq \sum_{\substack{\lambda \vdash n \\ \lambda \in \mathcal{D}}} \#\left\{\lambda_j \colon \lambda_j \equiv r \Mod t\right\},
\end{align*}
where $\mathcal{D}$ denotes the set of partitions with distinct parts, and he then proves that
\begin{equation}
    D(r,t;n) = \frac{3^{\frac{1}{4}}e^{\pi\sqrt{\frac{n}{3}}}}{2\pi tn^{\frac{1}{4}}}\left(\log(2) + \left(\frac{\sqrt{3}\log(2)}{8\pi} - \frac{t\pi}{4\sqrt{3}}\left(\frac{r}{t} - \frac{1}{2}\right)\right)n^{-\frac{1}{2}} + O(n^{-1})\right). \label{eq:craig-asym}
\end{equation}
Unlike \cref{eq:beckwith-asym}, notice that the $\psi$ function does not appear in this asymptotic. This difference arises from the properties of the different generating functions. 

Here, we generalize Craig's work by considering partitions where no part is repeated $k$ or more times. These are referred to as the \textit{$k$-regular partitions}. Notice that these are a direct generalization of distinct partitions, which are the $2$-regular partitions. In this case, the work of Hagis (see \cite[Corollary~4.1]{hagis}) yields an asymptotic formula for the number of $k$-regular partitions of $n$, which we denote by $p_k(n)$. Rewritten using the asymptotics for $I$-Bessel functions (see \cite[10.40.1]{nist}) we have for $K \coloneqq 1 - 1/k$ that
\begin{align*}
    p_k(n) \sim \left(\frac{K}{96k^2n^3}\right)^{\frac{1}{4}}e^{\pi\sqrt{\frac{2Kn}{3}}}.
\end{align*}

We extend and refine Craig's results to $k$-regular partitions for any $k \geq 2$. We thus define
\begin{align*}
	D_k(r,t;n) \coloneqq \sum_{\substack{\lambda \vdash n \\ \lambda \in \mathcal{D}_k}} \#\left\{\lambda_j \colon \lambda_j \equiv r \Mod t\right\},
\end{align*}
where $\mathcal{D}_k$ denotes the collection of all $k$-regular partitions. Note that, in our notation, we have $D(r,t;n)= D_2(r,t;n)$. We prove the following asymptotic formula for $D_k(r,t;n)$.
\begin{theorem}\label{thm:simple-asym}
	Let $t,k \geq 2$ and $1 \leq r \leq t$ be fixed integers. If $K \coloneqq 1 - 1/k$, then as $n \to \infty$ we have
	\[
	D_{k}(r,t;n) = \frac{3^{\frac{1}{4}}e^{\pi\sqrt{\frac{2Kn}{3}}}}{\pi t 2^{\frac{3}{4}}K^{\frac{1}{4}}n^{\frac{1}{4}}\sqrt{k}}\left(\log k + \left(\frac{3\sqrt{K}\log k}{8\sqrt{6}\pi} - \frac{t\pi(k-1)K^{\frac{1}{2}}}{2\sqrt{6}}\left(\frac{r}{t}- \frac{1}{2}\right)\right)n^{-\frac{1}{2}} + O(n^{-1})\right).
	\]
\end{theorem}

\begin{remark}
    In the $k=2$ case, we recover Theorem 1.1 of Craig in \cite{craig}. However, we refine Craig's methods to improve the relevant error terms in an explicit version of this theorem which is used to prove \Cref{thm:finite-check}. For a more detailed discussion, see the remark following \Cref{lemma:major-arc-xi}.
\end{remark}

\begin{example}
	Here we demonstrate the accuracy of the approximation of $D_k(r,t;n)$ in \Cref{thm:simple-asym}. Let $\widehat{D}_k(r,t;n)$ denote the asymptotic obtained in \Cref{thm:simple-asym} by ignoring all terms which are $O\left(n^{-\frac{5}{4}}e^{\pi\sqrt{\frac{2Kn}{3}}}\right)$, and let $Q_k(r,t;n) \coloneqq \frac{D_k(r,t;n)}{\widehat{D}_k(r,t;n)}$. The following table illustrates the convergence of $Q_k(r,t;n)$ to $1$ as $n \to \infty$.
\end{example}

\begin{figure}[ht]
    \centering
    \begin{tabular}{ | c |c | c | c | c | }
         \hline $n$ & 10 & 100 & 1000 & 10000 \\
         \hline
         $Q_3(1,4;n)$ & 1.02401 & 1.00616 &1.00249 & 1.00084 \\
         \hline
         $Q_3(2,4;n)$ & 1.06157 & 1.00469 & 1.00233 & 1.00083 \\
         \hline
         $Q_4(1,4;n)$ & 0.97589 & 1.01034 & 1.00401 & 1.00135 \\
         \hline
         $Q_4(2,4;n)$ & 0.97172 & 1.00691 & 1.00366 & 1.00131 \\
         \hline
    \end{tabular}
    \caption{Numerics for \Cref{thm:simple-asym}}
    \label{fig:simple-asym-numerics}
\end{figure}

Notice that the main term in the asymptotics given in \Cref{thm:simple-asym} and \cref{eq:beckwith-asym,eq:craig-asym} does not depend on $r$. Thus, asymptotically, the parts are equidistributed among congruence classes modulo $t$. That is if $P(n)$ (respectively $P_k(n)$) denotes the number of parts in partitions (respectively $k$-regular partitions) of $n$, then $\frac{T(r,t;n)}{P(n)} \to \frac{1}{t}$ and $\frac{D_k(r,t;n)}{P_k(n)} \to \frac{1}{t}$ as $n \to \infty$.
However, this equidistribution in the weak sense does not prohibit biases among the congruence classes; in fact, inspection of the lower order terms reveals the true structure of this bias. In particular, since $\psi\left(\frac{r}{t}\right)$ and $\frac{r}{t}$ are increasing functions of $r$ for $1 \leq r \leq t$, \cref{eq:beckwith-asym,eq:craig-asym} imply biases towards parts lying in lower congruence classes mod $t$ for parts in partitions of $n$ and for parts in $k$-regular partitions of $n$. Specifically, the asymptotics derived by Beckwith, Mertens, and Craig imply that $T(r, t;n) \geq T(s, t;n)$ and $D_2(r,t;n) \geq D_2(s,t;n)$ for $1 \leq r < s \leq t$ and for sufficiently large $n$. We obtain an analogous corollary.

\begin{corollary}\label{cor:eventual-ineq}
	Let $k,t \geq 2$ and $1 \leq r < s \leq t$, then we have that
	\begin{align*}
		D_k(r,t;n) - D_k(s,t;n) = \frac{e^{\pi\sqrt{\frac{2Kn}{3}}}}{4t2^{\frac{3}{4}}3^{\frac{1}{4}}K^{\frac{1}{4}}n^{\frac{1}{4}}\sqrt{k}}\left((s - r)n^{-\frac{1}{2}} + O(n^{-1})\right),
	\end{align*}
	and thus $D_k(r,t;n) \geq D_k(s,t;n)$ for sufficiently large $n$.
\end{corollary}

\begin{remark}
    Note that the simple identity 
    \[ 
    \prod_{n=1}^\infty \frac{1-x^{nk}}{1-x^n} = \prod_{n=1}^\infty (1 + \cdots + x^{n(k-1)})
    \]implies that the number $p_k(n)$ of $k$-regular partitions of $n$ is the same as the number of partitions of $n$ into parts which are not divisible by $k$, which we refer to as the $k$-indivisible partitions. Using the same techniques as in this paper, we may prove an asymptotic for the number of parts in $k$-indivisible partitions of $n$ which are congruent to $r$ mod $t$. For $k,t$ coprime, this once again implies equidistribution in the weak sense. However, the biases which emerge from the lower order terms are \textit{completely different} and much more difficult to analyze, due to the interaction between the $k$-multiplicative and $t$-multiplicative structure.  These biases will be addressed in a subsequent paper by the authors \cite{kindivis}. To illustrate the complexity of these biases, in \Cref{fig:indivis-biases} we give a table of all possible orderings of $\{1,\ldots,t\}$ induced by the biases in parts of $k$-indivisible partitions among congruence classes for $t = 7$. Notice that for $k = 2$, the congruence class $2 \Mod 7$ occurs in the fifth position. The bias against the residue class $2 \Mod 7$ may be explained by the fact that smallest allowed part which is $2 \Mod 7$ is $9$. Similarly, for $k = 6,13,10,20$ the transposition of $7,6 \Mod 7$ may be accounted for because the congruence class of $6 \Mod 7$ includes $6,13,20$ as its smallest members. However, for $k = 12$, there is a bias against the residue class $5 \Mod 7$ because its \textit{second} smallest member is excluded as a part in $12$-indivisible partitions. Thus the ordering is not only not induced by the ordering on integers, it is not even induced by the ordering on integers once we lift to the smallest allowed part in $k$-indivisible partitions.
\end{remark}

\begin{figure}[ht]
    \centering
    \begin{tabular}{ | c | c | c | c | c | c | c | c |}
        \hline 
        $k = 2$ & 1 & 3 & 5 & 7 & 2 & 4 & 6 \\
        \hline
        $k = 3$ & 1 & 2 & 4 & 5 & 7 & 3 & 6 \\ 
        \hline
        $k = 4$ & 1 & 2 & 3 & 5 & 6 & 7 & 4 \\
        \hline
        $k = 5$ & 1 & 2 & 3 & 4 & 6 & 7 & 5 \\
        \hline
        $k = 6,10,13,20$ & 1 & 2 & 3 & 4 & 5 & 7 & 6 \\
        \hline
        $k = 12$ & 1 & 2 & 3 & 4 & 6 & 5 & 7 \\
        \hline
        All other $k$  & 1 & 2 & 3 & 4 & 5 & 6 & 7 \\ \hline
    \end{tabular}
    \caption{Biases among congruence classes mod $t$ for $k$-indivisible partitions for $t = 7$, from most common to least common.}
    \label{fig:indivis-biases}
\end{figure}

\Cref{cor:eventual-ineq} confirms the heuristic that smaller parts occur more often than large ones in partitions of some fixed integer $n$, and these occupy the lower congruence classes mod $t$.
This argument does not apply directly for distinct partitions as in the work of Craig; however, it points to the general theme. Remarkably, numerics suggest that the biases for $T(r,t; n)$ and $D_k(r,t;n)$ have no counterexamples for $k \geq 3$, and that for $D_2(r,t;n)$ the counterexamples only occur for $n \leq 8$. This observation led Craig to develop an explicit version of \Cref{thm:simple-asym} for $k = 2$, which we extend to the case of all $k$-regular partitions (see \Cref{thm:effective} for details). As a result, we obtain the following theorem.

\begin{theorem}\label{thm:finite-check}
    Suppose that $2 \leq k \leq 10$, $2 \leq t \leq 10$, and $1 \leq r < s \leq t$. Then
    \begin{enumerate}
        \item The \textit{weak} inequality $D_k(r,t;n) \geq D_k(s,t;n)$ holds for all $n \geq 1$, $k\geq 3$ and $n \geq 9$, $k = 2$.
        \item The \textit{strong} inequality $D_k(r,t;n) > D_k(s,t;n)$ holds for all $n \geq 17$.
    \end{enumerate}
\end{theorem}

\begin{remark}
    As $t$ increases, there are stable counterexamples for the strong inequality, namely the counterexamples $(r,s, t, n) = (t-1, t, t, t+3)$ for $k=2$, $(t-1, t, t, t+2)$ for $k=3$ and $(t-1, t, t, t)$ for all other $k$. This occurs for the following reasons:
    
    \noindent (1) For $t$ at least 6, $k=2$ and $n=t+3$, the only occurrence of $t-1$ and $t$ in the partitions of $n$ is $(t, 3), (t, 2, 1), (t-1, 4), (t-1, 3, 1)$. Thus these counterexamples persist because $q(3) = q(4)$. \vspace{0.1in}
    
    \noindent (2) For $t$ at least 5, $k=3$ and $n=t+2$, the only occurrence of $t-1$ and $t$ in the partitions of $n$ is $(t, 2), (t, 1, 1), (t-1, 3), (t-1, 2, 1)$. Thus these persist because $p_3(2) = p_3(3)$. \vspace{0.1in}
    
    \noindent (3) For $t$ at least 4 and $n=t$, the only occurrence of $t-1$ and $t$ in the partitions of $n$ is $(t), (t-1, 1)$. Thus these persist because $p(0) = p(1)$.
\end{remark}
In light of this fact and the numerical checks the authors conjecture the following.
\begin{conjecture}
    If $t \geq 6$ and $1\leq r<s\leq t$, $k\geq 3$, then for $n\geq 1$ we have $D_k(r, t;n) \geq D_k(s, t;n)$. Further, for $k\geq 2$ if $n > \max\{t+3, 16\}$, the strict inequality $D_k(r, t;n)>D_k(s, t;n)$ holds.
\end{conjecture}

The proof of \Cref{thm:finite-check} relies on a finite computer search which we perform after making the error terms in \Cref{thm:simple-asym} explicit. To prove \Cref{thm:simple-asym}, we make use of a distinct variation on Hardy and Ramanujan's circle method originally due to Wright (see for example \cite{bringmannCircle,ngo,wright}). Unlike traditional applications of the circle method where modularity (specifically, a modular transformation law) is used to estimate generating functions near singularities, the relevant generating function for $D_k(r,t;n)$ is not modular. However, the generating function may be broken up into two components, the first of which is modular and whose transformation law can be used to obtain estimates. The second component may be expressed as a sum of polylogarithms, and classical Euler-Maclaurin summation yields another method for computing asymptotic expansions. These estimates may then be combined to produce the asymptotic expansion for $D_k(r,t;n)$ via Wright's circle method.


The paper is organized as follows. In \Cref{sec:prelim}, we recall known results which inform our approach, including a variant of Wright's circle method (see \Cref{thm:simple-circle}) and an explicit version of Euler-Maclaurin summation proven in \cite{craig} (see \Cref{prop:euler-maclaurin}). In \Cref{sec:estimates}, we derive a convenient form of the generating function for $D_k(r,t;n)$ and use it to obtain bounds on the major and minor arcs for use in the circle method. \Cref{sec:proofs-main} then goes through the application of these explicit asymptotics to prove \Cref{thm:simple-asym,thm:effective} along with \Cref{thm:finite-check}.

\section*{Acknowledgements}

The authors were participants in the 2022 UVA REU in Number Theory. They would like to thank Ken Ono, the director of the UVA REU in Number Theory, as well as their graduate student mentor William Craig. They would also like to thank their colleagues at the UVA REU for their encouragement and support. They are grateful for the support of grants from the National Science Foundation (DMS-2002265, DMS-2055118, DMS-2147273), the National Security Agency (H98230-22-1-0020), and the Templeton World Charity Foundation.

\section*{Data Availability}

The authors implemented a program in Mathematica to perform the finite checks at the end of the paper. This program can be obtained from GitHub at
\begin{center}
    \href{https://github.com/FayeAlephNil/KRegularBiases}{https://github.com/FayeAlephNil/KRegularBiases}
\end{center}
or, upon reasonable request, from the authors.

\section{Preliminaries}\label{sec:prelim}

\subsection{Bernoulli Polynomials and Polylogarithms}\label{subs:bernoulli-polylog}

In this subsection, we recall the \textit{Bernoulli polynomials} as well as \textit{polylogarithms} and several of their properties. The generating function for the Bernoulli polynomials is given in \cite[(24.2.3)]{nist} by
\begin{align*}
    \sum_{n \geq 0} B_n(x) \frac{t^n}{n!} \coloneqq \frac{te^{xt}}{e^t - 1}. 
\end{align*}
The \textit{Bernoulli numbers} are defined by $B_n \coloneqq B_n(0)$. Throughout the paper we freely use Lehmer's classical bound \cite{lehmer} on the size of Bernoulli polynomials for $x \in [0,1]$, which says for $n \geq 2$ that
\begin{align}
\abs{B_n(x)} \leq \frac{2\zeta(n)n!}{(2\pi)^n} \label{eq:lehmer},
\end{align}
where $\zeta(s) \coloneqq \sum_{n \geq 1} n^{-s}$ is the \textit{Riemann zeta function}. We also recall that $B_{2n + 1} = 0$ for $n > 0$. 

We now recall from \cite[25.12.10]{nist} that the \textit{polylogarithm} $\Li_s(q)$ for $s \in \CC$ is defined for $\abs{q} < 1$ by
\[
\Li_s(q) \coloneqq \sum_{n \geq 1} \frac{q^n}{n^s}
\]
and for other $q$ by analytic continuation. In particular, $\Li_0(q) = \frac{q}{1 - q}$ for $q \neq 1$. A simple consequence of this definition is that if $q = e^{-z}$, then
\begin{align}
    \frac{\partial}{\partial z} \Li_s(q) = -\Li_{s-1}(q) \label{eq:polylog-der}.
\end{align}
We require two primary results concerning polylogarithms. The first is a series expansion near $q = 1$ given in \cite[25.12.12]{nist}, whereby for $q = e^{-z}$ and $\abs{z} < 2\pi$ we have
\begin{align}
    \Li_s(q) = \Gamma(1-s)z^{s-1} + \sum_{n = 0}^\infty \zeta(s - n) \frac{(-z)^n}{n!} \label{eq:polylog-series}.
\end{align}
The second is a method of expressing $\Li_{-N}(q)$ for positive integer $N$ as a rational function of $q$. Namely, \cite{wood} gives
\begin{align}
\Li_{-N}(q) = \frac{1}{(1-q)^{N+1}} \sum_{m=0}^{N-1} \eulerian{N}{m} q^{N-m} \label{eq:polylog-eulerian},
\end{align}
where $\eulerian{N}{m}$ are the \textit{Eulerian numbers}, that is the number of permutations of $\{1,\ldots,N\}$ in which exactly $m$ elements are greater than the previous element.

\subsection{Classical Results on the Partition Generating Function}\label{subs:classic-partitions}

In this subsection, we recall the \textit{partition generating function} given by
\[
    \P(q) \coloneqq \sum_{n \geq 0} p(n)q^n = \prod_{n \geq 1} \frac{1}{1 - q^n},
\]
which converges absolutely for $\abs{q} < 1$. For later convenience, we recall the standard \textit{$q$-Pochhammer symbol} $(a;q)_\infty$ defined by
\[
    (a;q)_\infty \coloneqq \prod_{n \geq 1} \left(1 - aq^{n-1}\right).
\]
For $\abs{q} < 1$, we have $\P(q) = (q,q)_\infty^{-1}$. We require a few classical results concerning $\P(q)$ and $p(n)$. The first is an absolute bound on $\log \P(q)$ when $0< q < 1$ and the second is a subexponential bound on $p(n)$, together these can be found in \cite[Ch.~10.4.1-2]{stein} as exercises. Thus, for $0 < q < 1$ and $n \geq 1$, we have that
\begin{align}
    \log \P(q) &\leq \frac{\pi^2}{6(1-q)} \label{eq:logp-absolute}, \\
    p(n) &\leq e^{(\pi^2/6 + 1)\sqrt{n}} \label{eq:subexp-p}.
\end{align}
The bound \eqref{eq:subexp-p} is a crude version of Hardy-Ramanujan's asymptotic given in \eqref{eq:hi}. We also require Euler's pentagonal number theorem, which may be stated as
\begin{align}
    \P(q)^{-1} = 1 + \sum_{m \geq 1} (-1)^m\left(q^{m(3m+1)/2} + q^{m(3m-1)/2} \right) \label{eq:pentagonal}.
\end{align}
Finally, we make extensive use of the modular transformation law for $\P(q)$, which can be found in \cite[Thm.~5.1]{apostol} and which we state here for convenience:
\begin{align}
    \P(e^{-z}) = \sqrt{\frac z{2\pi}} \exp\left(\frac{\pi^2}{6z} - \frac z{24}\right) \P\left(e^{-\frac{4\pi^2}z}\right). \label{eq:modular-p}
\end{align}

\subsection{Euler-Maclaurin Summation}\label{subs:euler-maclaurin}

In this subsection, we recall classical Euler-Maclaurin summation, an asymptotic version due to Zagier \cite{zagier}, and an explicit version of this asymptotic due to Craig \cite{craig}. Euler-Maclaurin summation provides an exact formula for the difference between the integral $\int_a^b f(x) \d x$ and the finite sum $f(a + 1) + \cdots + f(b)$ when $a\leq b$ are positive integers. More specifically,
\begin{align*}
    \sum_{m=1}^{b-a} f(a + m) - \int_a^b f(x) \d x = \sum_{m=1}^N \frac{B_m}{m!}\left(f^{(m-1)}(b) - f^{(m-1)}(a)\right) + (-1)^{N+1}\int_a^b f^{(N)}(x)\frac{\widehat{B}_N(x)}{N!}\d x,
\end{align*}
where $\widehat{B}_n(x) \coloneqq B_n(x - \floor{x})$ and $\floor{x}$ denotes the greatest integer which is at most $x$. In \cite{zagier}, Zagier showed that when $f(z)$ has a known asymptotic expansion, the Euler-Maclaurin summation formula has a useful asymptotic variation. Here, we mean asymptotic expansion in the strong sense, where we say that $f(z) \sim \sum_{n \geq 0} b_nz^n$ provided that for all $N > 0$, we have $f(z) - \sum_{n=1}^N b_nz^n = O(z^{N+1})$ as $z \to 0$. This asymptotic form of Euler-Maclaurin summation has been applied in recent years to understand the growth functions without nice modular transformation laws (for examples, see \cite{beckwith2017,bringmannTauberian,bringmannRankUni,bringmannCircle,craig}).

For convenience, we now establish notation which we use freely for the remainder of the paper. For $\theta > 0$, define $D_\theta \coloneqq \{z \in \mathbb{C} : \abs{\operatorname{arg} z} < \frac{\pi}{2} - \theta\}$. If we set $z = \eeta + iy$ for $\eeta > 0$, then $z \in D_\theta$ if and only if $0 < \abs{y} < \Delta \eeta$ for some constant $\Delta > 0$ which depends on $\theta$. Furthermore set
\begin{align*}
    I_f \coloneqq \int_0^\infty f(x) \d x
\end{align*}
for any function $f$ for which this integral converges. Zagier requires that $f$ have \textit{rapid decay at infinity}, that is $f(x) = O(x^{-N})$ as $x \to \infty$ for any $N > 1$. The explicit version derived by Craig in \cite{craig} requires a less restrictive decay condition on $f(x)$ at infinity, referred to as \textit{sufficient decay}, which holds if $f(x) = O(x^{-N})$ as $x \to \infty$ for some $N > 1$. We now recall the original version due to Zagier (see \cite[Proposition~3]{zagier}). If $f$ has asymptotic expansion $f(x) \sim \sum_{n=0}^\infty c_nx^n$ at the origin and $f$ as well as all its derivatives have of rapid decay at infinity, then we have the asympototic expansion
\[
    \sum_{m = 1}^\infty f(mx) \sim \frac{I_f}{x} + \sum_{n=0}^\infty c_n\frac{B_{n+1}}{n+1}(-x)^n
\]
as $x \to 0^+$.

The following proposition is a refinement of this result due to Craig which we will use in the proof of \Cref{thm:effective}.
\begin{prop}[{\cite[Proposition~3.3]{craig}}]\label{prop:euler-maclaurin}
    Let $f(z)$ be $C^\infty$ in $D_\theta$ with power series expansion $f(z) = \sum_{n \geq 0} c_nz^n$ that converges absolutely in the region $0 \leq \abs{z} < R$ for some positive constant $R$, and let $f(z)$ and all its derivatives have sufficient decay as $z \to \infty$ in $D_\theta$. Then for any real number $0 < a \leq 1$ and any integer $N > 0$, we have
    \begin{align*}
        \abs{\sum_{m \geq 0} f((m+a)z) - \frac{I_f}{z} + \sum_{n = 0}^{N-1} c_n\frac{B_{n+1}(a)}{n+1}z^n} \leq \frac{M_{N+1}J_{f,N+1}(z)}{(N+1)!} \abs{z}^N + \sum_{k \geq N} \abs{c_k}\left(1 + \frac{k!}{10(k-N)!}\right) \abs{z}^k,
    \end{align*}
    where $M_{N+1} \coloneqq \max\limits_{0 \leq x \leq 1} \abs{B_{N+1}(x)}$ and
    \[
        J_{f,N+1}(z) \coloneqq \int_0^\infty \abs{f^{(N+1)}(w)} \abs{\d w},
    \]
    where the path of integration proceeds along the line through the origin and $z$.
\end{prop}

\subsection{Variants on Wright's Circle Method}\label{subs:wright-circle}

In this subsection, we recall a result of Bringmann, Craig, Ono, and Males from \cite{bringmannCircle}, which is a variation of Wright's circle method \cite{wright}. Wright's circle method allows one to give asymptotics for the coefficients of a $q$-series $F(q)$ which has suitable asymptotic behavior near the unit circle. More precisely, given a circle $\mathcal{C}$ centered at the origin with radius less than $1$ in the $q$-plane, we define its \textit{major arc} as that region of $\mathcal{C}$ where $F(q)$ is largest.  To define the major arc, let $\mathcal{L} = \{\eeta + iy \mid \abs{y} \leq \pi\}$, where $e^{-\eeta} < 1$ is the radius of $\mathcal{C}$. In our application and in those of \cite{beckwith2017,craig}, the major arc $\mathcal{C}_1$ is given by $\{e^{-z} \mid z \in \mathcal{L} \cap D_\theta\}$.
The \textit{minor arc} of $\mathcal{C}$ is then defined by $\mathcal{C}_2 \coloneqq \mathcal{C} \setminus \mathcal{C}_1$. In Wright's circle method, the integral taken over $\mathcal{C}_1$ gives the main term for the coefficients of $F(q)$ and the integral over $\mathcal{C}_2$ is an error term. We now recall the version of Wright's circle method which we will use in the proof of \Cref{thm:simple-asym}.

\begin{theorem}[{\cite[Proposition~4.4]{bringmannCircle}}]\label{thm:simple-circle}
    Suppose that $F(q)$ is analytic for $q = e^{-z}$ where $z = x + iy$ satisfies $x > 0$ and $\abs{y} < \pi$, and suppose that $F(q)$ has an expansion $F(q) = \sum_{n=0}^\infty c(n)q^n$ near $q = 1$. Let $N,\Delta > 0$ be fixed constants. Consider the following hypotheses:
    \begin{enumerate}[label={(\arabic*)}]
        \item\label{item:major-arc} As $z \to 0$ in the bounded cone $\abs{y} \leq \Delta x$ (major arc), we have
        \[
        F(e^{-z}) = Cz^Be^{\frac{A}{z}}\left(\sum_{j = 0}^{N-1} \alpha_jz^j + O_\theta(\abs{z}^N)\right),
        \]
        where $\alpha_s \in \CC, A,C \in \R^+,$ and $B \in \R$.
        \item\label{item:minor-arc} As $z \to 0$ in the bounded cone $\Delta x \leq \abs{y} < \pi$ (minor arc), we have
        \[
            \abs{F(e^{-z})} \ll_{\theta} e^{\frac{1}{\Re(z)}(A - \rho)},
        \]
        for some $\rho \in \R^+$.
    \end{enumerate}
    If \ref{item:major-arc} and \ref{item:minor-arc} hold, then as $n \to \infty$ we have
    \[
    c(n) = Cn^{\frac{1}{4}(-2B - 3)}e^{2\sqrt{An}}\left(\sum_{r=0}^{N-1} p_rn^{-\frac{r}{2}} + O\left(n^{-\frac{N}{2}}\right)\right),
    \]
    where $p_r \coloneqq \sum\limits_{j = 0}^r \alpha_j c_{j,r-j}$ and $c_{j,r} \coloneqq \frac{\left(-\frac{1}{4\sqrt{A}}\right)\sqrt{A}^{j + B + \frac{1}{2}}}{2\sqrt{\pi}} \cdot \frac{\Gamma(j + B + \frac{3}{2} + r)}{r!\Gamma(j + B + \frac{3}{2} - r)}$.
\end{theorem}

\begin{remark}
    The constant $C$ in \Cref{thm:simple-circle} does not appear in \cite{bringmannCircle}, but it is trivially equivalent to the result in \cite{bringmannCircle} by factoring out $C$ from each $\alpha_i$.
\end{remark}

\subsection{Estimates with Bessel Functions}\label{subs:bessel}

In this subsection, we establish estimates on the \textit{modified Bessel functions} $I_\nu(x)$ which we will require in our implementation of Wright's circle method. We recall that $I_\nu(x)$ is defined by
\[
    I_\nu(x) \coloneqq \left(\frac{x}{2}\right)^\nu \frac{1}{2\pi i}\int_{\mathcal{D}} t^{-\nu -1} \exp\left(\frac{x^2}{4t} + t\right) \d t,
\]
where $\mathcal{D}$ is any contour running from $-\infty$ below the real axis, counterclockwise around $0$, and back to $-\infty$ above the real axis. We shall choose $\mathcal{D} = \mathcal{D}_- \cup \mathcal{D}_0 \cup \mathcal{D}_+$, each of which depend on a particular choice of $\mu,\Delta > 0$, defined by
\begin{align*}
    \mathcal{D}_{\pm}^{\mu,\Delta} &\coloneqq \{u + iv \in \CC \mid u \leq \mu, v = \pm \Delta \mu\}, \\
    \mathcal{D}_0^{\mu,\Delta} &\coloneqq \{u + iv \in \CC \mid u = \mu, \abs{v} \leq \Delta \mu\}.
\end{align*}
Note that this dependence on $\mu,\Delta$ does not change the value of the integral, as we may freely shift the path of integration. We compare the size of $I_\nu(x)$ to the main term along $\mathcal{D}_0$. In particular, define
\[
    \widetilde{I}^{\mu,\Delta}_\nu(x) \coloneqq \left(\frac{x}{2}\right)^\nu \frac{1}{2\pi i}\int_{\mathcal{D}_0^{\mu,\Delta}} t^{-\nu -1} \exp\left(\frac{x^2}{4t} + t\right) \d t.
\]
The following lemma shows how $\widetilde{I}^{\mu,\Delta}_\nu(x)$ approximates $I_\nu(x)$.
\begin{lemma}\label{lemma:modified-bessel-estimate}
    Let $\nu, \mu,\Delta,x \in \R$, $\de = \sqrt{1+\Delta^2}$ and $\mu, \Delta>0$, then we have that
    \begin{align*}
        \abs{I_\nu(x) - \widetilde{I}^{\mu,\Delta}_\nu(x)} \leq  \frac{1}{\pi}\left(\frac{x}{2}\right)^\nu \exp\left(\frac{x^2}{4\mu\Delta^2}\right) \int_0^\infty \left(\delta x + u\right)^{-\nu - 1} \exp(-u) \d u.
    \end{align*}
\end{lemma}

\begin{proof}
    For ease of notation we suppress $\mu,\Delta$ where clear from context. We clearly have that
    \begin{align*}
        I_\nu(x) - \widetilde{I}_\nu(x) = \left(\frac{x}{2}\right)^{\nu}\frac{1}{2\pi i}\int_{\mathcal{D}_- \cup \mathcal{D}_+} t^{-\nu - 1}\exp\left(\frac{x^2}{4t} + t\right)\d t.
    \end{align*}
    For $t \in \mathcal{D}_-$ we may set $t = (\mu - u) - \Delta \mu i$. Since for $u \geq 0$, we have
    \[
        \Re\left(\frac{x^2}{4t}\right) = \frac{x^2}{4} \cdot \frac{\mu - u}{(\mu - u)^2 + \Delta^2 \mu^2} \leq \frac{x^2}{4\mu\Delta^2},
    \]
    it follows that
    \begin{align*}
        \abs{t^{-\nu - 1}\exp\left(\frac{x^2}{4t} + t\right)} &\leq \abs{t}^{-\nu - 1}\exp\left(\frac{x^2}{4\mu\Delta^2} - u\right) \leq \left(\delta\mu + u\right)^{-\nu - 1}\exp\left(\frac{x^2}{4\mu\Delta^2} - u\right).
    \end{align*}
    The same bound holds for $\mathcal{D}_+$. The result then follows.
\end{proof}

\section{Estimates on the Major/Minor Arc}\label{sec:estimates}

\subsection{Generating Functions}\label{subs:generating-func}

In this subsection, we derive a form of the generating function $\D_k(r,t;q)$ of $D_k(r,t;n)$ which is amenable to calculations. We then fit one component of the generating function to the framework of Euler-Maclaurin summation while relating the remaining component to the generating function for partitions. For the remainder of the paper, we use the notation $q \coloneqq e^{-z}$. Define
\begin{align*}
    \D_{k}(r,t;q) &\coloneqq \sum_{n \geq 0} D_{k}(r,t;n)q^n
\end{align*}
as the generating function of $D_k(r,t;n)$. We then have the following expression for $\D_k(r,t;q)$.

\begin{lemma}\label{lemma:d-gen-formula}
    We have that
    \[
        \D_{r,t}(k;q) = \frac{(q^k;q^k)_\infty}{(q;q)_\infty}\sum_{m \equiv r \Mod t} \frac{q^{m}}{1-q^{m}} - \frac{kq^{km}}{1 - q^{km}}.
    \]
\end{lemma}

\begin{proof}
    It is a classical fact that $\xi_k(q) = (q^k;q^k)_\infty(q;q)^{-1}_\infty$ is the generating function for the $k$-regular partitions. We now break the generating function up into pieces for each $m \equiv r \Mod t$. In particular, let $\mathcal{P}_k(m;q)$ be the generating function for $k$-regular partitions including $m$ as a part, weighting each partition by how many times it includes $m$. Then we have that
    \[
        \mathcal{P}_k(m;q) = \frac{(q^m + 2q^{2m} + \cdots +  (k-1)q^{(k-1)m})}{1 + q^m + q^{2m} + \cdots + q^{(k-1)m}} \cdot \frac{(q^k;q^k)_\infty}{(q;q)_\infty}.
    \]
    By simplifying
    \begin{align*}
        \sum_{j = 1}^{k-1} jq^{jm} &= \sum_{j=1}^{k-1} \sum_{\ell = j}^{k-1} (q^{jm} + \cdots + q^{(k-1)m}) = \sum_{j = 1}^{k-1} \frac{q^{jm}(1-q^{(k-j)m})}{1-q^m},
    \end{align*}
    we obtain 
    \begin{align*}
        \mathcal{P}_k(m;q) &= \frac{(q^{k};q^{k})_\infty}{(q;q)_\infty}  \sum_{j=1}^{k-1} \frac{q^{jm}(1-q^{(k-j)m})}{1-q^m} \cdot \frac{1-q^m}{1 - q^{mk}} = \frac{(q^{k};q^{k})_\infty}{(q;q)_\infty} \cdot \left(\frac{q^m}{1-q^m} - \frac{kq^{mk}}{1 - q^{mk}}\right).
    \end{align*}
    Summing over $m \equiv r \pmod t$ gives the desired result.
\end{proof}

For the remainder of this paper, we define the components\footnote{In Beckwith and Mertens work (see \cite{beckwith2017}), the analogue of $L_k$ is multiplied by $(2\pi)^{-1/2}q^{1/24}$ and the analogue of $\xi_k$ is multiplied by $(2\pi)^{1/2}q^{-1/24}$ to easily apply Ngo-Rhoades' variant of the circle method (see \cite{ngo}). The notation used in this paper matches that of Craig in \cite{craig}.} $\xi_k(q), L_k(r,t;q)$ of $\D_k(r,t;q)$ by
\begin{align*}
    \xi_k(q) \coloneqq \frac{(q^{k};q^{k})_\infty}{(q;q)_\infty}, && L_k(r,t;q) \coloneqq \sum_{m\equiv r \Mod t} \frac{q^{m}}{1-q^{m}} - \frac{kq^{km}}{1 - q^{km}}.
\end{align*}
Now define $E_k(z) \coloneqq \frac{e^{-z}}{1 - e^{-z}} - \frac{ke^{-zk}}{1 - e^{-zk}}$. Recalling the expansion of $\Li_0(q)$, we may write 
\[
    E_k(z) = \Li_0(q) - k \Li_0(q^k).
\]
We now see that $L_k(r,t;q)$ may be expressed as a sum over integers of $E_k(z)$ evaluated at specific values.

\begin{lemma}\label{lemma:lrt-ek}
    We have
    \[
        L_k(r,t;q) = \sum_{\ell \geq 0} E_k((\ell t + r)z).
    \]
\end{lemma}

\begin{proof}
    Immediate from the definitions of $L_k(r,t;q)$ and $E_k(z)$.
\end{proof}

\Cref{lemma:lrt-ek} will later allow us to use Euler-Maclaurin summation to estimate $L_k(r,t;q)$ on the major arc (see \Cref{lemma:major-arc-lrt}). This follows the methods established by Beckwith and Mertens as well as Craig (see \cite{beckwith2017} and \cite{craig}) to estimate the appropriate analogue $L_k(r,t;q)$ on the major arc. In order to apply \Cref{prop:euler-maclaurin} to $L_k(r,t;q)$, we develop the following expression in terms of polylogarithms for $E_k^{(N)}(z)$, as well as a series expansion near zero.

\begin{lemma}\label{lemma:ek-N-series}
    We have that
    \begin{align*}
        E_k^{(N)}(z) &= (-1)^N\left(\Li_{-N}(q) - k^{N+1}\Li_{-N}\left(q^k\right)\right) = \sum_{m=0}^\infty \frac{(1-k^{m+N+1})B_{N+m+1}}{(N+m+1) \cdot m!}z^m. 
    \end{align*}
    where the second equality only holds when $\abs{z} < \frac{2\pi}{k}$.
\end{lemma}

\begin{proof}
    The first equality follows from \cref{eq:polylog-der} because $E_k(z) = \Li_0(q) - k\Li_0(q^k)$. Then, substituting the series expansions for $\Li_{-N}(q)$ and $\Li_{-N}(q^k)$ given in \cref{eq:polylog-series}, we have
    \[
        E_k^{(N)}(z) = (-1)^N\sum_{m=0}^\infty \zeta(-N-m)\frac{(-z)^m(1-k^{m+N+1})}{m!},
    \]
    when $\abs{z} < \frac{2\pi}{k}$. We now recall that the Riemann zeta function's value at negative integers may be expressed in terms of the Bernoulli Numbers. In particular, $\zeta(-N-m) = (-1)^{N+m}\frac{B_{N+m+1}}{N+m+1}$.
    This gives the claimed formula above.
\end{proof}

\begin{remark}
    Note that $E_k(z)$ has rapid decay within any region $D_\theta$, since $\Li_0(q) = \frac{e^{-z}}{1 - e^{-z}}$ and all its derivatives have rapid decay in $D_\theta$. Furthermore, note that $E_k^{(N)}(z)$ does not have a pole at $z = 0$, since the principal parts of $\Li_{-N}(q)$ and $k^{N+1}\Li_{-N}(q^k)$ cancel. This fact accounts for the shape of the secondary tames, namely, the absence of the digamma function $\psi$.
\end{remark}

Specializing \Cref{lemma:ek-N-series} to $N = 0$, we have a series expansion for $E_k$ near zero, by which we define constants $e_{k,m}$ below:
\begin{align}
    E_k(z) &= \sum_{m \geq 0} \frac{e_{k,m}}{m!}z^m \coloneqq \sum_{m \geq 0} \frac{(1-k^{m+1})B_{m+1}}{(m+1) \cdot m!}z^m \label{eq:ekm}.
\end{align}

We also note here for later use the transformation law we obtain for $\xi_k(q)$ using the transformation law for $\P(q)$.
\begin{lemma}\label{lemma:xi-transform}
    For $q = e^{-z}$ and $\e \coloneqq \exp\left(-\frac{4\pi^2}{kz}\right)$, we have that
    \[
        \xi_k(q) = \frac{1}{\sqrt{k}}\exp\lr{\frac{\pi^2}{6z}\lr{1 -\frac{1}{k}} + \frac{z}{24}(k-1)}\frac{\P(\e^k)}{\P(\e)}.
    \]
\end{lemma}

\begin{proof}
        Note that $\xi_k(q) = \frac{\P(q)}{\P(q^k)}$. Therefore, using the modular transformation law for $\P(q)$ given in \cref{eq:modular-p} we have that
        \begin{align*}
        \xi_k(q) &= \frac{\P(q)}{\P(q^k)} = \frac{\sqrt{\frac z{2\pi}} \exp\left(\frac{\pi^2}{6z} - \frac z{24}\right) \P\left(\exp\left(-\frac{4\pi^2}{z}\right)\right)}
{\sqrt{\frac {zk}{2\pi}} \exp\left(\frac{\pi^2}{6zk} - \frac {zk}{24}\right) \P\left(\exp\left(-\frac{4\pi^2}{zk}\right)\right)} \\
&= \frac{1}{\sqrt{k}}\exp\lr{\frac{\pi^2}{6z}\lr{1 -\frac{1}{k}} + \frac{z}{24}(k-1)}\frac{\P(\e^k)}{\P(\e)} \tag*{\vspace{0.05in}\qedhere}.
        \end{align*}
\end{proof}
We will use this transformation law extensively to estimate $\xi_k(q)$ on the major and minor arcs. This differs significantly from the methods of Craig in \cite{craig} as he uses Euler-Maclaurin summation to estimate $\log \xi_2(q)$ (for more details, see the remark following \Cref{lemma:major-arc-xi}).

\subsection{Explicit Bounds on the Major Arc}

In this subsection, we obtain explicit bounds on the functions $L_k(r,t;q)$ and $\xi_k(q)$ on the major arc. For convenience, let $\Delta > 0$ be fixed, and set $\de \coloneqq \sqrt{1 + \Delta^2}$. Furthermore, note that for $z = \eeta + iy$ in the region $0 \leq \abs{y} \leq \Delta \eeta$ the hypothesis $\eeta < \frac{\pi}{kt\de}$ is equivalent to $\abs{z} < \frac{\pi}{kt}$.

\begin{lemma}\label{lemma:major-arc-lrt}
    Let $0 < r \leq t$ be integers and $z = \eeta + iy$ a complex number satisfying $0 \leq \abs{y} < \Delta \eeta$ as well as $\abs{z} \leq \frac{\pi}{kt}$, then we have that
    \begin{align*}
        &\abs{L_k(r,t; e^{-z}) - \frac{\log k}{tz} + \frac{k-1}{2}B_1\left(\frac r t \right) - \frac{k^2-1}{24}B_2\left(\frac r t \right)tz + \frac{k^4-1}{2880}B_4\left(\frac r t \right)t^3z^3 } \\
        &\leq \frac{1.94\de^7 k^6}{30240\pi^6}\abs{tz}^5 + \frac{0.1216k^6}{30240}\abs{tz}^5 + 0.0412k^6\abs{tz}^5,
    \end{align*}
    where $\de = \sqrt{1 + \Delta^2}$.
\end{lemma}


\begin{proof}
    This proof applies \Cref{prop:euler-maclaurin} to $E_k(z)$, which has radius of convergence $2\pi/k$. We note that $M_6 = \frac 1 {42}$. Thus applying \Cref{prop:euler-maclaurin} to $E_k(z)$ with $a = r/t$ we have
    \begin{align*}
        &\abs{\sum_{m \geq 0} E_k\left(\left(m + \frac{r}{t}\right)z\right) - \frac{I_{E_k}}{z} + \frac{k-1}{2}B_1(a) - \frac{k^2-1}{24}B_2(a)z + \frac{k^4-1}{2880}B_4(a)z^3 } \\
        &\leq \frac{J_{E_k,6}(z)}{30240}\abs{z}^5 + \sum_{m \geq 5} \left(1 + \frac{m!}{10(m-5)!}\right)\frac{\abs{e_{k,m}}}{m!}\abs{z}^m.
    \end{align*}
    We also compute that 
    \[
        I_{E_k} = \int_0^\infty \frac{e^{-z}}{1 - e^{-z}} - \frac{ke^{-zk}}{1-e^{-zk}} \d z = -\log\left(\frac{1 - e^{-z}}{1 - e^{-zk}}\right)_{z = 0}^{z = \infty} = \log k.
    \]
    Substituting $tz$ for $z$ and putting this together with \Cref{lemma:lrt-ek} yields
    \begin{align*}
        &\abs{L_k(r,t; e^{-z}) - \frac{\log k}{tz} + \frac{k-1}{2}B_1(a) - \frac{k^2-1}{24}B_2(a)tz + \frac{k^4-1}{2880}B_4(a)t^3z^3 } \\
        &\leq \frac{J_{E_k,6}(z)}{30240}\abs{tz}^5 + \sum_{m \geq 5} \left(1 + \frac{m!}{10(m-5)!}\right)\frac{\abs{e_{k,m}}}{m!}\abs{tz}^m.
    \end{align*}
    
    
    
    We now bound the integral $J_{E_k,6}(z)$ far from zero, using \Cref{lemma:ek-N-series} and properties of polylogarithms. Specifically we bound the following, where $\alpha = \frac{\pi z}{k\abs{z}}$:
    \begin{align*}
        \abs{\int_\alpha^\infty E_k^{(6)}(w) \d w } &\leq \int_\alpha^\infty \abs{E_k^{(6)}(w)} \abs{\d w} \leq \int_\alpha^\infty \abs{\Li_{-6}(e^{-w})} \abs{\d w} + k^7 \int_\alpha^\infty\abs{\Li_{-6}(e^{-wk})} \abs{\d w}.
    \end{align*}
    We now may use the expression of $\Li_{-6}(e^{-w})$ as a rational function of $e^{-w}$ from \cref{eq:polylog-eulerian} and the triangle inequality to see that for $x \coloneqq \Re w$ we have
    \begin{align*}
      \abs{\Li_{-6}(e^{-w})} &\leq \frac{e^{x} + 57 e^{2x} + 302 e^{3x} + 302 e^{4x} + 57 e^{5x} + e^{6x}}{(e^{x} - 1)^7}.
    \end{align*}
    Integrating this with $\beta \in \mathbb{R}_{>0}$ as the lower bound of the integral yields
    \begin{align*}
        \int_\beta^\infty \frac{e^{x} + 57 e^{2x} + 302 e^{3x} + 302 e^{4x} + 57 e^{5x} + e^{6x}}{(e^{x} - 1)^7} \d x \leq 0.04(34 + 27\cosh(\beta) + \cosh(2\beta))\csch^6(\beta/2).
    \end{align*}
    Changing variables to $w = ue^{i\theta}$ where $\arg z = \theta$ then to $x = u\cos \theta$ we then have that
    \begin{align*}
        \int_\alpha^\infty \abs{\Li_{-6}(e^{-w})} \abs{\d w} \leq \frac{1}{\cos \theta}\int_{\abs{\alpha}\cos \theta}^\infty \frac{e^{x} + 57 e^{2x} + 302 e^{3x} + 302 e^{4x} + 57 e^{5x} + e^{6x}}{(e^{x} - 1)^7} \d x.
    \end{align*}
    Therefore, since $\cos \theta \geq 1/\de$ for $z = \eeta + iy$ satisfying $\abs{y} < \Delta \eeta$, we compute
    \begin{align*}
        \int_\alpha^\infty \abs{\Li_{-6}(e^{-w})} \abs{\d w} &\leq  0.04\de(34 + 27\cosh(\abs{\alpha}/\de) + \cosh(2\abs{\alpha}/\de))\csch^6(\abs{\alpha}/2\de) \\
        &\leq 4.54\de\csch^6(\abs{\alpha}/2\de) \leq \frac{4.54\cdot 2^6\de^7k^6}{\pi^6}.
    \end{align*}
    Similarly, with a simple change of variables $\omega = wk$ we have that
    \begin{align*}
        \int_\alpha^\infty\abs{\Li_{-6}(e^{-wk})} \abs{\d w} &= \frac{1}{k}\int_{k\alpha} \abs{\Li_{-6}(e^{-\omega})} \abs{\d \omega} \\
        &\leq 0.04\de(34 + 27\cosh(\pi/\de) + \cosh(2\pi/\de))\csch^6(\pi/2\de) \leq \frac{24.6\cdot 2^6 \de^7}{k\pi^6}.
    \end{align*}
    Combining these two calculations yields the following estimate for the integral away from zero:
    \begin{align*}
        \abs{\int_\alpha^\infty E^{(6)}(w) \d w} \leq \frac{29.14 \cdot 2^6\de^7k^6}{\pi^6} \leq 1.94\de^7 k^6.
    \end{align*}
    For the near 0 part we use the series expansion of $E_k^{(6)}$ near zero from \Cref{lemma:ek-N-series}, which gives
\begin{align*}
\int_0^{\alpha} \abs{E_k^{(6)}(w)}\abs{\d w}
&=
\int_0^{\alpha}
\sum_{m=0}^\infty \frac{|B_{m+7}|\cdot (k^{m+7} - 1)\cdot |w|^m}{(m+7)\cdot m!}\abs{\d w}.
\end{align*}
Applying Lehmer's Bound \eqref{eq:lehmer} and that $B_{m+7}$ vanishes for $m$ even, we see that
\begin{align*}
\int_0^{\alpha} \abs{E_k^{(6)}(w)}\abs{\d w} &\leq
\int_0^{\alpha}
\sum_{m \text{ odd}}
\frac{2(m+7)!\zeta(m+7)\cdot (k^{m+7} - 1)}{(2\pi)^{m+7}(m+7)\cdot m!}\abs{w}^m \abs{\d w}\\
&\leq
\int_0^{\alpha}
\sum_{m \text{ odd}}
\frac{2(m+1)\dots(m+6)\cdot \zeta(8)\cdot k^7}
{(2\pi)^7}\left(\frac{k\abs{w}}{2\pi}\right)^m \abs{\d w}.
\end{align*}
Because $E_k^{(6)}(z)$ has no pole at $z = 0$, we use Fubini's theorem to exchange the sum and the integral to obtain that
\begin{align*}
\int_0^{\alpha} \abs{E_k^{(6)}(w)}\abs{\d w} &\leq
\sum_{m \text{ odd}}
\frac{2\zeta(8)k^7}{(2\pi)^7}
\int_0^{\alpha}
(m+1)\dots(m+6)\left(\frac{k\abs{w}}{2\pi}\right)^m \abs{\d w}\\
&=
\frac{2\zeta(8)k^7}{(2\pi)^7}
\sum_{m \text{ odd}}
\int_0^{\frac12}
(m+1)\dots(m+6)u^m\frac{2\pi\d u}{k} \\
&=
\frac{2\zeta(8)k^6}{(2\pi)^6}
\sum_{m \text{ odd}}
(m+2)\dots(m+6)\left(\frac12 \right)^{m+1} \leq
0.1216k^6,
\end{align*}
where we use the change of variables $u = \frac{k\abs w}{2\pi}$ and the final sum has been calculated numerically.
    
    We now bound the sum over coefficients appearing on the right hand side of the inequality. Using \cref{eq:ekm,eq:lehmer} we have that
    \begin{align*}
        \abs{\frac{e_{k,m}}{m!}} &\leq \frac{\zeta(m+1)(1+k^{m+1})}{(2\pi)^{m+1}} \leq \frac{\pi^2(1+k^{m+1})}{6(2\pi)^{m+1}} \leq \frac{\pi^2}{3} \left(\frac{k}{2\pi}\right)^{m+1}.
    \end{align*}
    Therefore, we have
    \begin{align*}
        \sum_{m \geq 5} \left(1 + \frac{m!}{10(m-5)!}\right)\frac{\abs{e_{k,m}}}{m!}\abs{tz}^{m-5} &\leq \frac{\pi^2k^6}{3(2\pi)^6} \sum_{m \geq 5} \left(1 + \frac{m!}{10(m-5)!}\right)\left(\frac{k\abs{tz}}{2\pi}\right)^{m-5} \\
        &\leq \frac{\pi^2k^6}{3(2\pi)^6} \sum_{m \geq 5} \left(1 + \frac{m!}{10(m-5)!}\right)\frac{1}{2^{m-5}} \leq 0.0412k^6.
    \end{align*}
    Combining these bounds for the sum and the integral, we see that
    \begin{align*}
        &\abs{L_k(r,t; e^{-z}) - \frac{\log k}{tz} + \frac{k-1}{2}B_1(a) - \frac{k^2-1}{24}B_2(a)tz + \frac{k^4-1}{2880}B_4(a)t^3z^3 } \\
        &\leq \frac{1.94\de^7 k^6}{30240\pi^6}\abs{tz}^5 + \frac{0.1216k^6}{30240}\abs{tz}^5 + 0.0412k^6\abs{tz}^5. \tag*{\qedhere}
    \end{align*}
\end{proof}
For later use, we also provide a bound on $\abs{L_k(r,t;e^{-z})}$ on the major arc.
\begin{corollary}\label{cor:major-arc-lrt-absolute}
    Let $0 < r \leq t$ be integers and $z = \eeta + iy$ a complex number satisfying $0 \leq \abs{y} < \Delta \eeta$ as well as $\eeta < \frac{\pi}{kt\de}$, then we have
    \[
        \abs{L_k(r,t;e^{-z})} \leq \frac{41+\log k}{\abs{tz}} + \frac{1.94\de^7}{30240\abs{tz}}.
    \]
\end{corollary}
\vspace{-0.2in}
\begin{proof}
    We know from \Cref{lemma:major-arc-lrt} that
    \begin{align*}
        \abs{L_k(r,t;e^{-z})} &\leq \frac{\log k}{t\abs{z}} + \frac{k-1}{2}\abs{B_1\left(\frac r t \right)} + \frac{k^2 - 1}{24}\abs{B_2\left(\frac r t \right)}t\abs{z} + \frac{k^4 - 1}{2880}B_4\left(\frac r t \right)t^3\abs{z}^3 \\
        &\quad \quad + \frac{1.94\de^7 k^6}{30240\pi^6}\abs{tz}^5 + \frac{0.1216k^6}{30240}\abs{tz}^5 + 0.0412k^6\abs{tz}^5.
    \end{align*}
    Using the trivial bound on $B_1(x) = x - \frac12$ and Lehmer's bound we see that
    \begin{align*}
        \abs{L_k(r,t;e^{-z})} &\leq \frac{\log k + \frac{k-1}{4}\abs{tz} + \frac{k^2-1}{144}\abs{tz}^2 + \frac{k^4 - 1}{86400}\abs{tz}^4 + \frac{1.94\de^7 k^6}{30240\pi^6}\abs{tz}^6 + \frac{0.1216k^6}{30240}\abs{tz}^6 + 0.0412k^6\abs{tz}^6}{\abs{tz}}.
    \end{align*}
    Applying the restriction $\abs{tz} < \frac{\pi}{k}$ we have that
    \begin{align*}
        \abs{L_k(r,t;e^{-z})}&< \frac{41 + \log k}{\abs{tz}} + \frac{1.94\de^7}{30240\abs{tz}}. \tag*{\qedhere}
    \end{align*}
\end{proof}
Before we bound $\xi_k$ on the major arc, we require an elementary lemma concerning sums of exponentials which may be bounded by geometric series.
\begin{lemma}\label{lemma:mvt-geom}
    Suppose we have a series of the form $\sum_{m \geq b} e^{f(x)}$ such that $f'(x)$ is decreasing and $f'(b) < 0$ then
    \[
        \sum_{m \geq b} e^{f(x)} \leq \frac{e^{f(b)}}{1 - e^{f'(b)}}.
    \]
\end{lemma}

\begin{proof}
    We have by the mean value theorem that for each $m \geq b$, $f(m+1)-f(m) = f'(c) \leq f'(b)$ for $c$ between $m$ and $m + 1$. Therefore, using the geometric series formula,
    \begin{align*}
        \sum_{m \geq b} e^{f(x)} \leq \sum_{m \geq b} e^{f(b)}e^{(m-b)f'(b)} = \frac{e^{f(b)}}{1-e^{f'(b)}}. \tag*{\qedhere}
    \end{align*}
\end{proof}
Applying this along with \Cref{lemma:xi-transform} and Euler's pentagonal number theorem (see \eqref{eq:pentagonal}) yields the following bound.
\begin{lemma}\label{lemma:major-arc-xi}
    Let $z = \eeta + iy$ be a complex number satisfying $0 \leq \abs{y} \leq \Delta \eeta$ with $0 \leq \eeta < \frac{4\pi^2}{2.35\de}$, then
    \begin{align*}
    \abs{\xi_k(q) - \Phi_k(z)} \leq 7\abs{\Phi_k(z)}\cdot \abs{\exp\left(\frac{-4\pi^2}{kz}\right)}
    \end{align*}
    where
    \begin{align*}
        \Phi_k(z) \coloneqq \frac{1}{\sqrt{k}}\exp\left(\frac{\pi^2}{6z}\left(1 - \frac{1}{k}\right) + \frac{z}{24}(k-1)\right).
    \end{align*}
\end{lemma}

\begin{proof}
    Recalling that $\e = e^{-4\pi^2/kz}$, \Cref{lemma:xi-transform} gives
    \begin{align*}
        \xi_k(q) = \Phi_k(z) \cdot \frac{\P(\e^k)}{\P(\e)} = \Phi_k(z) \cdot \left(1 + \sum_{m \geq 1} (-1)^m\left(\e^{\frac{m(3m+1)}{2}} + \e^{\frac{m(3m-1)}{2}}\right)\right)\left(1 + \sum_{m \geq 1} p(m)\e^{km}\right).
    \end{align*}
    Expanding using Euler's pentagonal number theorem (see \eqref{eq:pentagonal}) yields
    \begin{align*}
        \frac{\xi_k(q) - \Phi_k(z)}{\Phi_k(z)} &= \sum_{m \geq 1} (-1)^m\left(\e^{m(3m+1)/2} + \e^{m(3m-1)/2}\right) \cdot \left(1 + \sum_{m \geq 1} p(m)\e^{km}\right) \\
        &\quad\quad + \sum_{m \geq 1} p(m)\e^{km}\left(1 + \sum_{m \geq 1}(-1)^m\left(\e^{m(3m+1)/2} + \e^{m(3m-1)/2}\right)\right).
    \end{align*}
    We bound each of the sums above individually. The fact that
    \begin{align*}
        \prod_{m \geq 1} (1 - \e^m) &= 1 + \sum_{m \geq 1} (-1)^m\left(\e^{m(3m+1)/2} + \e^{m(3m-1)/2}\right)
    \end{align*}
    gives
    \begin{align*}
        \abs{-1 + \prod_{m \geq 1} (1 - \e^m)} &= \abs{\sum_{m \geq 1} (-1)^m\left(\e^{m(3m+1)/2} + \e^{m(3m-1)/2}\right)} \leq \abs{\e} + \sum_{m \geq 2} \abs{\e}^m &= \abs{\e} + \frac{\abs{\e}^2}{1-\abs{\e}}.
    \end{align*}
    When $\eeta < \frac {2\pi^2}{\delta^2k}$, we have
    \begin{align*}
        \abs\e = \exp\lr{-\Re\lr{\frac{4\pi^2}{kz}}} = \exp\lr{-\frac{4\pi^2\eeta}{k\abs z^2}} \leq \exp\lr{-\frac{4\pi^2}{k\delta^2\eeta}} \leq \exp(-2) \leq \frac17.
    \end{align*}
    Thus we have
    \[
    \frac{\abs{\e}^2}{1-\abs\e} \leq 1.17\abs\e^2.
    \]
    Using \eqref{eq:subexp-p}, we may write
    \begin{align*}
        \abs{\sum_{m \geq 1} p(m) \e^{km}} \leq \sum_{m \geq 1} e^{2.7\sqrt{m}} \abs{\e}^{km} = \sum_{m \geq 1} e^{2.7\sqrt{m} - 4\pi^2m\Re(1/z)}.
    \end{align*}
    Now applying $\Re(1/z) = \frac{\eeta}{\abs{z}^2} \geq \frac{1}{\de\eeta}$ and differentiating with respect to $m$, we see that
    \[
        \frac{\d}{\d m} \left(2.7\sqrt{m} - 4\pi^2m\Re(1/z)\right) = \frac{1.35}{\sqrt{m}} - 4\pi^2\Re\left(\frac{1}{z}\right) \leq 1.35 - \frac{4\pi^2}{\de\eeta} \leq -1.
    \]
    Therefore we may apply \Cref{lemma:mvt-geom} to get $\abs{\P(\e^k) - 1} \leq 23.6\abs{\e}^k$.
    Combining these, we obtain
    \begin{align*}
    \abs{\frac{\P(\e^k)}{\P(\e)} - 1} &\leq |\e| + \frac{\abs \e^2}{1 - \abs\e} + 23.6 \abs\e^k + 23.6\abs\e^k\lr{\abs\e + \frac{\abs\e^2}{1 - \abs\e}} \\
    &\leq \abs\e + 25.6\abs\e^2 + 11\abs\e^2 \leq 7\abs\e,
    \end{align*}
    which gives
    \[
    \abs{\xi_k(q) - \Phi_k(z)} \leq 7\abs{\Phi_k(z)}\cdot \abs{\exp\lr{\frac{-4\pi^2}{kz}}}
    \]
    as desired.
\end{proof}

\begin{remark}
    This bound differs significantly from that derived by Craig in \cite{craig}. Specifically, Craig bounds $\abs{\log \xi_2(q) - \log \Phi_2(q)}$ by $35\abs{z}^2$ using Euler-Maclaurin summation. In contrast, using modularity of $\xi_k$, we obtain exponential savings on the error term in the lemma above. For a direct comparison to the bound derived by Craig for $\abs{\xi_2(q) - \Phi_2(q)}$, see Corollary 4.4 of \cite{craig}.
\end{remark}

\subsection{Explicit Bounds on the Minor Arc}

We now proceed to bound $\xi_k(e^{-z})$ and $L_k(r,t;e^{-z})$ on the minor arc, where for $z = \eeta + iy$ we have $\Delta \eeta \leq \abs{y} < \pi$.

\begin{lemma}\label{lemma:minor-arc-xi}
    Let $t \geq 2$ be an integer. Assume that $z = \eeta + iy$ satisfies $\Delta \eeta \leq \abs{y} \leq \pi$ and $0 < \eeta < \frac{12 \log k}{k}$ (which is always true when $\eeta < \frac\pi{kt\de}$), then we have that
    \[
        \abs{\xi_k(e^{-z})} \leq 3\exp\left(\frac{\pi^2K}{6\eeta}\left(\frac{1}{2} + \frac{3}{\pi^2} + \frac{6}{\pi^2\de^2}\right)\right).
    \]
\end{lemma}
\begin{proof}
    Using the formula obtained in \Cref{lemma:xi-transform}, we have
    \begin{align*}
        \xi_k(e^{-z}) &= \sqrt{\frac{1}{k}}\exp\left(\frac{\pi^2}{6z}\left(1 - \frac{1}{k}\right) + \frac{z}{24}(k-1)\right)\frac{\P(e^{-4\pi^2/z})}{\P(e^{-4\pi^2/kz})}.
    \end{align*}
    Taking absolute values then yields
    \begin{align*}
        \abs{\xi_k(e^{-z})} &= \sqrt{\frac{1}{k}}\exp\left(\frac{\pi^2\Re(1/z)}{6}\left(1 - \frac{1}{k}\right) + \frac{\eeta}{24}(k-1)\right)\frac{\abs{\P(e^{-4\pi^2/z})}}{\abs{\P(e^{-4\pi^2/kz})}} \\
        &\leq \sqrt{\frac{1}{k}}\exp\left(\frac{\pi^2\Re(1/z)}{6}\left(1 - \frac{1}{k}\right) + \frac{\eeta}{24}(k-1)\right)\P(e^{-4\pi^2\Re(1/z)})S_k(z),
    \end{align*}
    where
    \begin{align*}
        S_k(z) &\coloneqq 1 + e^{-4\pi^2\Re(1/kz)} + \frac{e^{-8\pi^2\Re(1/kz)}}{1 - e^{-4\pi^2\Re(1/kz)}} \geq \frac{1}{\P(e^{-4\pi^2/kz})}.
    \end{align*}
    For the lower bound $S_k(z) \geq (\P(e^{-4\pi^2/kz}))^{-1}$, we use the pentagonal number theorem in a similar way to the technique seen in \Cref{lemma:major-arc-xi}. Now we have that
    \begin{align*}
        \log \abs{\frac{\xi_k(e^{-z})}{S_k(z)}} = -\frac{1}{2}\log k + \frac{\pi^2\Re(1/z)}{6}\left(1 - \frac{1}{k}\right) + \frac{\eeta}{24}(k-1) + \log \P(e^{-4\pi^2\Re(1/z)}).
    \end{align*}
    For $\mu > 0$, we have $ \Re(1/z) = \frac{\eeta}{\abs{z}^2} \leq \frac{1}{\eeta \de^2}$ and
    \begin{align*}
        \log \P(e^{-\mu}) &= \sum_{m \geq 1} \frac{e^{-m\mu}}{m(1-e^{-m\mu})} \leq \frac{\pi^2e^{-\mu}}{6(1-e^{-\mu})} \leq \frac{\pi^2}{6\mu},
    \end{align*}
    where the second line follows from \cref{eq:logp-absolute}. Using these bounds, we obtain
    \begin{align*}
        \log \abs{\frac{\xi_k(e^{-z})}{S_k(z)}} &\leq -\frac{\log k}{2} + \frac{\pi^2}{6\eeta\de^2}\left(1 - \frac{1}{k}\right) + \frac{\eeta}{24}(k-1) + \frac{\abs{z}^2}{24\eeta} \\
        &\leq-\frac{\log k}{2} + \frac{\pi^2}{24\eeta}\left(1 + \frac{4(1-1/k)}{\de^2}\right) + \frac{k\eeta}{24} \leq \frac{\pi^2}{24\eeta}\left(1 + \frac{4(1-1/k)}{\de^2}\right).
    \end{align*}
    We now trivially bound $S_k(z)$ to see that
    \begin{align*}
        S_k(z) &\leq 2 + \frac{e^{-4\pi^2 \Re(1/kz)}}{4\pi^2\Re(1/kz)} \leq 2 + e^{\frac{1}{4\pi^2\Re(1/kz)}} \leq 2 + e^{1/(4\eeta)} \leq 3e^{1/(4\eeta)}.
    \end{align*}
    Therefore, we have that
    \begin{align*}
        \abs{\xi_k(e^{-z})} \leq 3\exp\left(\frac{\pi^2}{24\eeta}\left(1 + \frac{4(1-1/k)}{\de^2} + \frac{6}{\pi^2}\right)\right).
    \end{align*}
    Using that $k \geq 2$ quickly yields the second inequality in the statement above.
\end{proof}

\begin{lemma}\label{lemma:minor-arc-lrt}
    Let $k,t \geq 2$ and $1 \leq r \leq t$. Assuming that $z = \eeta + iy$ satisfies $0 < \eeta < \frac{1}{k}$, we have
    \[
        \abs{L_k(r,t;e^{-z})} \leq \frac{3.1}{\eeta^2}.
    \]
\end{lemma}

\begin{proof}
    By the triangle inequality we immediately have that
    \[
        \abs{L_{r,t}(k;q)} \leq \sum_{m \geq 1} \frac{\abs{q}^m}{1 - \abs{q}^m} + k\sum_{m \geq 1} \frac{\abs{q}^{mk}}{1 - \abs{q}^{mk}}.
    \]
    We may then write that
    \begin{align*}
        \sum_{m \geq 1} \frac{\abs{q}^m}{1 - \abs{q}^m} &= \sum_{m \geq 1} \left(\abs{q}^m + \abs{q}^{2m} + \cdots\right) = \sum_{m \geq 1} \sigma_0(m) \abs{q}^m \leq \sum_{m \geq 1} m\abs{q}^m = \frac{\abs{q}}{(1 - \abs{q})^2},
    \end{align*}
    where $\sigma_0(m) \leq m$ is the number of positive divisors of $m$. Using that $\abs{q} = \abs{e^{-z}} = e^{-\eeta}$ we then have that
    \begin{align*}
        \abs{L_{r,t}(k;q)} &\leq \frac{\abs{q}}{(1 - \abs{q})^2} + \frac{k\abs{q}^k}{(1 - \abs{q}^k)^2} = \frac{e^{-\eeta}}{(1 - e^{-\eeta})^2} + \frac{ke^{-k\eeta}}{(1 - e^{-k\eeta})^2} \leq \frac{e^\eeta}{\eeta^2} + \frac{e^{k\eeta}}{k\eeta^2} < \frac{3.1}{\eeta^2}
    \end{align*}
    using the conditions that $k \geq 2, 0 < \eeta < 1/k \leq 1/2$ to bound $1/k$ as well as $e^\eeta,e^{k\eeta}$.
\end{proof}

\section{Proofs of the Main Results}\label{sec:proofs-main}

\subsection{Proof of \texorpdfstring{\Cref{thm:simple-asym}}{Theorem 1.1}}\label{subs:simple-asym}

We now apply the explicit asymptotics obtained in the previous section to prove \Cref{thm:simple-asym} using Wright's circle method. Specifically, we apply the variant stated in \Cref{thm:simple-circle} and proved in \cite{bringmannCircle}. To do so, we must estimate $\D_k(r,t;e^{-z}) = \xi_k(e^{-z})L_k(r,t;e^{-z})$ on both the major arc, where $z = \eeta + iy$ satisfies $\abs{y} < \Delta \eeta$, and on the minor arc, where $\Delta \eeta \leq \abs{y} \leq \pi$. Specifically, we need to show that
\begin{align}
    \D_k(r,t;e^{-z}) &= \frac{1}{\sqrt{k}}z^{-1}e^{\frac{\pi^2K}{6z}}\left(\frac{\log k}{t} + \frac{k-1}{2}\left(\frac{1}{2} - \frac{r}{t}\right)z + O_\Delta(\abs{z}^2)\right)\label{eq:asym-major}
\end{align}
and
\begin{align}
    \abs{\D_k(r,t;e^{-z})} &\ll_\Delta e^{\frac{1}{\Re(z)}\left(\frac{\pi^2K}{6} - \rho\right)} \label{eq:asym-minor}
\end{align}
on the major and minor arc, respectively, where $\rho > 0$. These correspond to conditions \ref{item:major-arc} and \ref{item:minor-arc} in \Cref{thm:simple-circle} respectively.

On the major arc, by \Cref{lemma:major-arc-lrt} and \Cref{lemma:major-arc-xi} we have that 
\begin{align*}
    L_k(r,t;e^{-z}) &= \frac{\log k}{tz} - \frac{k-1}{2}\left(\frac{r}{t} - \frac{1}{2}\right) + O(z)
\end{align*}
and
\begin{align*}
    \xi_k(e^{-z}) &= \frac{1}{\sqrt{k}}\exp\left(\frac{\pi^2K}{6z}\right)\left(1 + O\left(\exp\left(\frac{-4\pi^2}{kz}\right)\right)\right) = \frac{1}{\sqrt{k}}\exp\left(\frac{\pi^2K}{6z} + O(z)\right).
\end{align*}
Together, these imply \cref{eq:asym-major}. On the minor arc, \Cref{lemma:minor-arc-lrt,lemma:minor-arc-xi} imply \cref{eq:asym-minor}, as for sufficiently large $\de > 1$ (that is sufficiently large $\Delta$) we have $\frac{1}{2} + \frac{3}{\pi^2} + \frac{6}{\pi^2\de^2} < 1$. Thus we may apply \Cref{thm:simple-circle} to $\D_k(r,t;q)$, which gives the claimed asymptotic formula.

\subsection{Making \texorpdfstring{\Cref{thm:simple-asym}}{Theorem 1.1} Explicit}\label{subs:effective}

To prove \Cref{thm:finite-check}, we must first make \Cref{thm:simple-asym} explicit so that we can reduce the problem to a simple finite computer check. We follow Ngo and Rhoades' proof in \cite{ngo} of a variant on Wright's circle method applicable to a wide range of functions.

\begin{theorem}\label{thm:effective}
 	Let $\de > 1$ be some fixed real number. For any fixed integers $t,k \geq 2$, $1 \leq r \leq t$ and all integers $n > \frac{\de^2k^2t^2}{6}$, we have
 	\begin{align*}
 		\abs{D_k(r,t;n) - \alpha_0^tV_0(k;n) - \alpha_1^{r,t}V_1(k;n) - \alpha_2^{r,t}(V_2(k;n) - \alpha_4V_4(k;n)}  \leq \mathcal{E}_k(t,\de;n),
 	\end{align*}
 	where the $V_s$ are certain contour integrals defined in \eqref{eq:vs}, and where
 	\begin{align*}
 		\mathcal{E}_k(t,\de;n) &\coloneqq C_1n\exp\left(\sqrt{n}\left(\frac{0.52}{\sqrt{K}} + C_1'(\de)\sqrt{K}\right)\right) + C_2(\de,k,t)\sqrt{n}\exp\left(\pi\sqrt{\frac{2Kn}{3}}\left(1 - C_2'(\de,k)\right)\right) \\
 		&\quad \quad + \frac{C_3(\de,k,t)}{n^3}\exp\left(\pi\sqrt{\frac{2Kn}{3}}\right),
 	\end{align*}
 	for constants $\alpha_j^{r,t},C_j,C_j'$ which are determined explicitly in \eqref{eq:alphas}, \eqref{eq:eff-e1}, \eqref{eq:eff-e2}, and \eqref{eq:eff-e3} within the proof below. Furthermore, if $\de > \sqrt{\frac{1.29K}{1.28K - 0.52}}$, then 
 	\[
 	    \Err_k(r,t;n) = O_{t,k,\de}\left(n^{-3}\exp\left(\pi\sqrt{\frac{2Kn}{3}}\right)\right).
 	\]
\end{theorem}

\begin{proof}
    For convenience, we let $K \coloneqq 1 - 1/k$. Now let $\mathcal{C}$ be the circle in the complex $q$-plane with center $0$ and radius $e^{-\eeta}$ where $\eeta \coloneqq \pi \sqrt{\frac{K}{6n}}$. By Cauchy's formula and \Cref{lemma:d-gen-formula}, we have
    \[
        D_k(r,t;n) = \frac{1}{2\pi i}\int_{\mathcal{C}} \frac{\mathcal{D}_{r,t}(k;q)}{q^{n+1}} \d q = \frac{1}{2\pi i}\int_{\mathcal{C}} \frac{L_{r,t}(k;q)\xi_k(q)}{q^{n+1}} \d q.
    \]
    Throughout, we fix $q = e^{-z}$ with $z = \eeta + iy$. Let $\mathcal{L}$ be the line segment in $z$-plane corresponding to $\mathcal{C}$. Let $\mathcal{C}_1$ be the major arc of $\mathcal{C}$ wherein the corresponding $z = \eeta + iy \in \mathcal{L}_1$ satisfies $0 \leq \abs{y} < \Delta \eeta$. Likewise we call $\mathcal{C}_2 \coloneqq \mathcal{C} \setminus \mathcal{C}_1$ the minor arc. We also assume $\eeta < \frac{\pi}{kt\de}$, which is equivalent to $n > \frac{\de^2k^2t^2}{6}$.
    
    Define for $s \geq 0$ the integrals
    \begin{align}
        V_s(k;n) &\coloneqq \frac{1}{2\pi i }\int_{\mathcal{C}_1} \frac{z^{s-1}}{q^{n+1}} \Phi_k(z) \d q = \frac{1}{2\pi i\sqrt{k}} \int_{\eeta - i\Delta\eeta}^{\eeta + i\Delta\eeta} z^{s-1}\exp\left(\frac{\pi^2K}{6z} + \frac{z(k-1)}{24} + nz\right) \d z. \label{eq:vs} 
    \end{align}
    By changing variables to $z$ and keeping track of orientation, we will see in \Cref{subs:finite-check} that this is related to the main term in the Bessel function
    \[
        I_{-s}\left(\pi\sqrt{\frac{2K}{3}\left(n + \frac{k-1}{24}\right)}\right).
    \]
    Thus it will suffice to approximate $D_k(r,t;n)$ in terms of the integrals $V_s(k;n)$. We have the decomposition
    \[
        D_k(r,t;n) - \alpha_0^t(k)V_0(k;n) - \alpha_1^{r,t}(k)V_1(k;n) - \alpha_2^{r,t}(k)V_2(k;n) - \alpha_4^{r,t}(k)V_4(k;n) = E_1 + E_2 + E_3,
    \]
    where the constants $\alpha_j^{r,t}$ are defined by
    \begin{align}
        \alpha_0^t(k) &\coloneqq \frac{\log k}{t}, && \alpha_1^{r,t}(k) \coloneqq -\frac{k - 1}{2}B_1\left(\frac{r}{t}\right), \notag \\
        \alpha_2^{r,t}(k) &\coloneqq \frac{k^2 - 1}{24}B_2\left(\frac{r}{t}\right), && \alpha_4^{r,t}(k) \coloneqq -\frac{k^4 - 1}{2880}B_4\left(\frac{r}{t}\right), \label{eq:alphas}
    \end{align}
    and the integrals $E_1,E_2,E_3$ are defined by
    \begin{align*}
        E_1 &\coloneqq \frac{1}{2 \pi i}\int_{\mathcal{C}_2} \frac{L_{r,t}(k;q)\xi_k(q)}{q^{n+1}} \d q, && E_2 \coloneqq \frac{1}{2 \pi i}\int_{\mathcal{C}_1} \frac{L_{r,t}(k;q)(\xi_k(q) - \Phi_k(z))}{q^{n+1}} \d q ,
        \end{align*}
    and
    \begin{align*}
        E_3 &\coloneqq \frac{1}{2\pi i}\int_{\mathcal{C}_1} \frac{(L_{r,t}(k;q) - \alpha_0z^{-1} - \alpha_1 - \alpha_2z - \alpha_4z^3)\Phi_k(z)}{q^{n+1}} \d q.
    \end{align*}
    For convenience, we shall suppress the dependence of $\alpha_s^{r,t}(k)$ on $r,t,k$ where these are are clear from context. We now bound the error terms $E_1, E_2, E_3$.
    
    We first estimate $\abs{E_1}$. To do so, we first rewrite $\abs{q^{-n}}$ in terms of $n$ as $\abs{q^{-n}} = \exp\left(\pi \sqrt{\frac{Kn}{6}}\right)$. Now, we bound $\abs{E_1}$ using the max-length bound along with \Cref{lemma:minor-arc-lrt,lemma:minor-arc-xi}, which together imply
    \begin{align}
        \abs{E_1} &\leq \abs{L_{r,t}(k;q)}\abs{\xi_k(q)}\abs{\exp(nz)} \leq \frac{9.3}{\eeta^2}\exp\left(\frac{\pi^2}{24\eeta}\left(1 + \frac{6}{\pi^2} + \frac{4K}{\de^2}\right) + \frac{\sqrt{K}\pi \sqrt{n}}{\sqrt{6}}\right) \notag \\
        &\leq 12n\exp\left(\sqrt{n}\left(\frac{0.52}{\sqrt{K}} + \pi\sqrt{\frac{K}{6}} + \frac{1.29\sqrt{K}}{\delta^2}\right)\right) \eqqcolon \Err^1_k(t,\de;n). \label{eq:eff-e1}
    \end{align}
    We now estimate $E_2$. Similar to above, we switch to the $z$ plane and use the max-length bound, \Cref{cor:major-arc-lrt-absolute}, and \Cref{lemma:major-arc-xi}. Define
    \[
        A_{t, k,\de} \coloneqq \frac{41 + \log k}{t} + \frac{1.94\de^7}{30240t} < \frac{41 + \log k + 0.00007\de^7}{t} \eqqcolon A_{t,k,\de}^\ast.
    \]
    Thus, we have
    \begin{align*}
        \abs{E_2} &\leq
        \frac{1}{2\pi} \abs{L_{r, t}(k;q)}\abs{\xi_k(q) - \Phi_k(z)} \abs{e^{nz}}\cdot 2\Delta \eeta \\
        &\leq \frac{7\Delta\eeta}{\pi} \abs{\frac{A_{t, k,\de}}{z}}\abs{\Phi_k(z)}\cdot \abs{\exp\left(-\frac{4\pi^2}{kz}\right)}e^{n\eeta} \\
        &\leq \frac{7\Delta\sqrt{n}}{4\sqrt{K}}A_{t,k,\de}^\ast\exp\lr{\pi\sqrt{\frac{2Kn}{3}}\lr{1 - \frac{12}{\de^2(k-1)}}} \exp\lr{{\frac{\eeta(k-1)}{24}}}.
    \end{align*}
    From the bound $\eeta < \frac{\pi}{kt\delta}$, we have
    \[
    \frac{\eeta(k-1)}{24} < \frac{\pi}{t\delta 24} < \frac{\pi}{48\cdot \de} < \frac{0.066}{\de} \leq 0.066,
    \]
    which implies
    \begin{equation}
    \abs{E_2} \leq \frac{2\Delta\sqrt{n}}{\sqrt{K}}\left(\frac{41 + \log k + 7\cdot 10^{-5}\de^7}{t}\right)\exp\lr{\pi\sqrt{\frac{2Kn}{3}}\lr{1 - \frac{12}{\de^2(k-1)}}} \label{eq:eff-e2}.
    \end{equation}
    To bound $E_3$, first we see that applying \Cref{lemma:major-arc-lrt} on $\mathcal{C}_1$ with some computation yields
    \begin{align*}
        \abs{L_{r,t}(k;q) - \alpha_0z^{-1} - \alpha_1 - \alpha_2z - \alpha_4z^3}  \leq \left(\frac{6.68\de^7k^6}{10^8}+ 0.0413k^6\right)t^5\de^5\eeta^5.
    \end{align*}
    We also have that
    \begin{align*}
        \abs{\Phi_k(z)} &= \frac{1}{\sqrt{k}}\exp\left(\frac{\pi^2K}{6}\Re(1/z) + \frac{\eeta(k-1)}{24}\right) \leq  \frac{1.07}{\sqrt{k}}\exp\left(\pi\sqrt{\frac{Kn}{6}}\right).
    \end{align*}
    Therefore, we may use the max-length bound on $\abs{E_3}$ to see that
    \begin{align}
        \abs{E_3} &\leq \frac{4.8\eeta}{2\pi}\abs{L_{r,t}(k;q) - \alpha_0z^{-1} - \alpha_1 - \alpha_2z - \alpha_4z^3}\abs{\Phi_k(z)}\abs{\exp(nz)} \notag \\
        &\leq \left(\frac{2.44\de^7k^6}{10^7}+ 0.151k^6\right)\frac{t^5K^3\de^5}{n^3\sqrt{k}}\exp\left(\pi\sqrt{\frac{2Kn}{3}}\right) \eqqcolon \Err^3_k(t,\de;n). \label{eq:eff-e3}
    \end{align}
    We have therefore shown that
    \[
        \abs{D_k(r,t;n) - \alpha_0V_0(k;n) - \alpha_1V_1(k;n) - \alpha_2V_2(k;n) - \alpha_4V_4(k;n)} \leq \Err_k(t,\de;n),
    \]
    where
    \begin{align*}
        \Err_k(t,\de;n) \coloneqq \Err^1_k(t,\de;n) + \Err^2_k(t,\de;n) + \Err^3_k(t,\de;n).
    \end{align*}
    The asymptotic for $\Err_k(t,\de;n)$ follows by choosing $\de$ large enough so that $\Err^3_k$ dominates $\Err^1_k$ and $\Err^2_k$.
\end{proof}

\vspace{-0.2in} 

\subsection{Proof of \texorpdfstring{\Cref{thm:finite-check}}{Theorem 1.3}}\label{subs:finite-check}

We now use \Cref{thm:effective} to prove \Cref{thm:finite-check} by reducing the problem to a finite computer search. Throughout, we assume that $\de > \sqrt{\frac{1.29K}{1.28K - 0.52}}$ so that the error term from \Cref{thm:effective} will in fact be smaller than the main term. To this end, define
\[
    \de_{\min}(k) \coloneqq \sqrt{\frac{1.29K}{1.28K - 0.52}} > \sqrt{\frac{1.29}{0.76}} > 1.3.
\]
Note the bound from below follows because this function of $K$ is maximized by taking $K = 1$ in the region $k \geq 2$. 

Before we apply \Cref{thm:effective}, we must relate $V_s(k;n)$ to $\widetilde{I}_{-s}(x)$. For convenience, define
\begin{align*}
    u &\coloneqq z\left(n + \frac{k-1}{24}\right), && x \coloneqq \pi\sqrt{\frac{2K}{3}\left(n + \frac{k - 1}{24}\right)} && \mu \coloneqq \eeta\left(n + \frac{k-1}{24}\right).
\end{align*}
We now change variables from $z$ to $u$ in the $V_s(k;n)$ integral to obtain the main term of a Bessel function evaluated at $x$. In particular, we have that
\begin{align*}
    V_s(k;n) &= \frac{1}{2\pi i\sqrt{k}}\int_{\eeta - i\Delta\eeta}^{\eeta + i\Delta\eeta} z^{s - 1}\exp\left(\frac{x^2}{4u} + u\right) \d z = \frac{\left(n + \frac{k-1}{24}\right)^{-s}}{2\pi i\sqrt{k}}\int_{\mu - i\Delta\mu}^{\mu + i\Delta\mu} u^{s-1}\exp\left(\frac{x^2}{4u} + u\right) \d u \\
           &= \frac{\pi^{s-1}\left(\sqrt{\frac{2K}{3}}\right)^{s/2}\left(n + \frac{k-1}{24}\right)^{-s/2}}{\sqrt{k}}\widetilde{I}^{\mu,\Delta}_{-s}(x).
\end{align*}
From here, we use our explicit estimates to show $D_k(r,t;n) > D_k(s,t;n)$ for some integers $n \geq 0$, $1 \leq r < s \leq t$. Because $D_k(r,t;n) -D_k(s,t;n) = \sum_{j=r}^{s-1} D_k(j,t;n) - D_k(j+1,t;n)$, we need only show the inequality $D_k(r,t;n) > D_k(r+1,t;n)$ for each $1 \leq  r < t$. For brevity, define $\beta^{r,t}_j(k) \coloneqq \alpha^{r,t}_j(k) - \alpha^{r+1,t}_j(k)$ and
\begin{align*}
    M_{r,t}(k;n) \coloneqq \alpha_0V_0(k;n) + \alpha_{1}^rV_1(k;n) + \alpha_2^rV_2(k;n) + \alpha_r^rV_4(k;n).
\end{align*}
Applying \Cref{thm:effective} to $D_k(r,t;n)$ and $D_{r+1,t}(k;n)$ it suffices to show that
\[
M_{r,t}(k;n) - M_{r+1,t}(k;n) \geq 2\Err_k(t,\de;n),
\]
which is equivalent to
\begin{equation}
    \beta_1^{r,t}(k)V_1(k;n) + \beta_2^{r,t}V_2(k;n) + \beta_4^{r,t}V_4(k;n) > 2\Err_k(t,\de;n). \label{eq:needed-v-form}
\end{equation}
We thus need to evaluate $\beta_j^{r,t}(k)$ for $j = 1,2,4$. Using that $B_1(x) = x - \frac 1 2, B_2(x) = x^2 - x + \frac 1 6,$ $B_4(x) = x^4 - 2x^3 + x^2 - \frac{1}{30},$ and $1 \leq r < t$, we have
\begin{align*}
    \beta_1^{r,t}(k) &= \frac{k-1}{2t}, && \beta_2^{r,t}(k) = \frac{(k^2 - 1)(t - 2r - 1)}{24t^2},
\end{align*}
and
\begin{align*}
    \beta_4^{r,t}(k) &= \frac{(k^4 - 1)(4r^3 - 6r^2(t-1) + (t-1)^2 + 2r(t^2 - 3t + 2))}{2880t^4}.
\end{align*}
From \Cref{lemma:modified-bessel-estimate}, we have
\begin{align*}
    \abs{I_{-s}(x) - \widetilde{I}_{-s}(x)}
    &\leq \frac{1}{\pi}\left(\frac{x}{2}\right)^{-s} \exp\left(\frac{x^2}{4\mu\Delta^2}\right) \int_0^\infty \left(\delta \mu + u\right)^{s - 1} \exp(-u) \d u.
\end{align*}
Let $s$ be a positive integer. Using partial integration, we see that
\begin{align*}
    \int_0^\infty (c+u)^{s-1}\exp(-u) \d u 
    &= c^{s-1} + (s-1)\int_0^\infty(c+u)^{s-2}e^{-u}\d u
\end{align*}
for any constant $c$. Applying this repeatedly, notice that this iterative process terminates when $s = 1$, and thus we have that
\[
\int_0^\infty \left(\delta \mu + u\right)^{s - 1} \exp(-u) \d u = (\delta\mu)^{s-1}+(s-1)(\delta\mu)^{s-2}+\dots + (s-1)!.
\]
Combining the following equality
\[
\frac{\delta\mu}{x} = \frac{\delta\frac{\pi}{2}\sqrt{\frac{2K}{3n}}(n+\frac{k-1}{24})}{\pi\sqrt{\frac{2K}{3}(n+\frac{k-1}{24})}} = \frac{\delta}2 \sqrt{\frac{n+\frac{k-1}{24}}{n}} =
\frac{\delta}2 \sqrt{1 + \frac{k-1}{24n}}
\]
and the bound $\eeta < \frac{\pi}{kt\delta}$, which implies
\[
n > \frac{k(k-1)t^2\delta^2}{6},
\]
we have
\[
\frac{\delta\mu}{x} < \frac{\delta}{2}\sqrt{1 + \frac{1}{16k\delta^2}} < 0.6\delta.
\]
Thus, because $x > 3$ for $n \geq 4$, we have
\begin{align*}
    \left(\frac x 2\right)^{-s}\int_0^\infty (c+u)^{s-1}\exp(-u) \d u &=
    \left(\frac{x}{2}\right)^{-s} \int_0^\infty \left(\delta \mu + u\right)^{s - 1} \exp(-u) \d u \\
    &<
    \frac{2^s}{x}\lr{(0.6\delta)^{s-1}+(s-1)\frac{(0.6\delta)^{s-2}}{3} +\dots +(s-1)!\frac{1}{3^{s-1}}}.
\end{align*}
We may then explicitly evaluate $\exp\left(\frac{x^2}{4\mu\Delta^2}\right)$ as
\[
    \exp\left(\frac{x^2}{4\mu\Delta^2}\right) = \exp\left(\frac{\pi}{2\Delta^2}\sqrt{\frac{2Kn}{3}}\right).
\]
Combining these bounds and the computation above with \Cref{lemma:modified-bessel-estimate} yields that for $s=1,2,4$ we have
\begin{align*}
    \abs{I_{-1}(x) - \widetilde{I}_{-1}(x)} &\leq
    \frac{2}{\pi x}\exp\left(\frac{\pi}{2\Delta^2}\sqrt{\frac{2Kn}{3}}\right) \eqqcolon \mathcal{I}_{-1}(n,\de), \\
    \abs{I_{-2}(x) - \widetilde{I}_{-2}(x)} &\leq
    \frac{4}{\pi x}\exp\left(\frac{\pi}{2\Delta^2}\sqrt{\frac{2Kn}{3}}\right)(0.6\delta + \frac{0.6^2\delta^2}{3}) \eqqcolon \mathcal{I}_{-2}(n,\de),
\end{align*}
and
\begin{align*}
    \abs{I_{-4}(x) - \widetilde{I}_{-4}(x)} &\leq
    \frac{16}{\pi x}\exp\left(\frac{\pi}{2\Delta^2}\sqrt{\frac{2Kn}{3}}\right)((0.6\delta)^3 + 3\frac{(0.6\delta)^2}{3}+ 6\frac{0.6\delta}{9} + 6\frac{1}{27}) \eqqcolon \mathcal{I}_{-4}(n,\de).
\end{align*}
Now define
\begin{align*}
    W_s(k;n) \coloneqq \frac{\pi^{s-1}\left(\sqrt{\frac{2K}{3}}\right)^{s/2}\left(n + \frac{k-1}{24}\right)^{-s/2}}{\sqrt{k}} \geq 0,
\end{align*}
so that $V_s(k;n) = W_s(n)\widetilde{I}_{-s}(x)$. Then to prove \cref{eq:needed-v-form} it suffices to show
\begin{align}
    &\beta_1 W_1(k;n)I_{-1}(x) + \beta_2W_2(k;n)I_{-2}(x) + \beta_4W_4(k;n)I_{-4}(x) \notag \\
    &> 2\Err_k(t,\de;n) + \beta_1W_1(k;n)\mathcal{I}_{-1}(n,\de) + \beta_2W_2(k;n)\mathcal{I}_{-2}(n,\de) + \beta_4W_4(k;n)\mathcal{I}_{-4}(n,\de) \label{eq:needed-i-form}.
\end{align}
An analysis of the relevant asymptotics for Bessel functions yields that this holds for sufficiently large $n$, as \cite[10.40.1]{nist} indicates that $I_{-s}(x)$ has main term $(4s^2-1)^2\frac{e^x}{\sqrt{2\pi x}}$ for any $s$, and since $W_s(k;n)$ is of order $n^{-s/2}$, we see that $W_1(k;n)I_{-1}(x)$ dominates the $I$-Bessel functions above. The remaining terms are error terms asymptotically because when $\de \geq \de_{\min}(k)$ we have that
\[
    \Err_k(t,\de;n) = O\left(n^{-3}\exp\left(\pi\sqrt{\frac{2Kn}{3}}\right)\right).
\]
Furthermore, since $\de > 1.3$, we have that $\Delta > \sqrt{\frac{1}{2}}$, so that
\begin{align*}
    W_s(k;n)\mathcal{I}_{-s}(n,\de) &= O\left(x^{-1}\exp\left(\frac{\pi}{2\Delta^2}\sqrt{\frac{2Kn}{3}}\right)\right)
\end{align*}
is an error term. Finally, note that $\beta_1$ is positive, and therefore the inequality holds for sufficiently large $n$.

In summary, we have shown in order to show that $D_k(r,t;n) > D_k(s,t;n)$ for all $0 < r < s \leq t$ for a fixed value of $n$, it suffices to consider the case $s = r + 1$ and show that the inequality \eqref{eq:needed-i-form} holds for all $0 < r < t$, so long as $n > \max\left\{\frac{\de^2k^2t^2}{6},4\right\}$ (note the maximum is irrelevant unless $k=t=2$). In light of this, we define the integer $N_k(t,\de)$ as the smallest positive integer such that both $N_k(t,\de) > \frac{\de^2k^2t^2}{6}$ the inequality \eqref{eq:needed-i-form} is true for all $n > N_k(t,\de)$ and for all $1 \leq r \leq t - 1$. We may then define\footnote{In \cite{craig}, Craig uses the notation $N_t(n)$ for $N_2(t,\sqrt{101})$} $N_k(t) \coloneqq \min\limits_{\de > \de_{\min}(k)} N_k(t,\de)$. The table in \Cref{fig:finite-check-numerics} gives upper bounds on $N_k(t)$ for $2 \leq t \leq 10$ and $2 \leq k \leq 10$ which have been computed with the aid of a computer.

\begin{figure}[ht]
    \centering
    \begin{tabular}{|c|c|c|c|c|c|c|c|c|c|}
    \hline $t$ & 2 & 3 & 4 & 5 & 6 & 7 & 8 & 9 & 10 \\
    \hline
    $N_2(t) \leq $ & 27633 \linecomment{$\de = 8.6$} & 31342 \linecomment{$\de = 7.25$} & 35554 \linecomment{$\de = 6.3$} & 40711 \linecomment{$\de = 5.6$} & 47067 \linecomment{$\de = 5.05$} & 54736 \linecomment{$\de = 4.6$} & 63726 \linecomment{$\de = 4.3$} & 74091 \linecomment{$\de = 4$} & 85823 \linecomment{$\de = 3.8$}\\
    \hline
    $N_3(t) \leq$ & 4718 \linecomment{$\de = 3.95$} & 7140 \linecomment{$\de = 3.05$} & 10540 \linecomment{$\de = 2.65$} & 15051 \linecomment{$\de = 2.4$} & 20820 \linecomment{$\de =2.25$} & 28031 \linecomment{$\de=2.15$} & 36415 \linecomment{$\de = 2.05$} & 46604 \linecomment{$\de = 2$} & 58354 \linecomment{$\de = 1.95$} \\
    \hline
    $N_4(t) \leq$ & 4130 \linecomment{$\de = 2.8$} & 7430 \linecomment{$\de = 2.3$} & 12294 \linecomment{$\de = 2.05$} & 19131 \linecomment{$\de = 1.9$} &  27862 \linecomment{$\de = 1.85$} & 39330 \linecomment{$\de=1.8$} & 52735 \linecomment{$\de = 1.75$} & 68212 \linecomment{$\de = 1.7$} & 89465 \linecomment{$\de = 1.7$} \\
    \hline
    $N_5(t) \leq$ & 4624 \linecomment{$\de = 2.35$} & 9243 \linecomment{$\de = 2$} & 16234 \linecomment{$\de = 1.82$} & 26051 \linecomment{$\de =1.73$} & 39085 \linecomment{$\de=1.67$} & 55653 \linecomment{$\de = 1.63$} & 76162 \linecomment{$\de = 1.6$} & 100588 \linecomment{$\de = 1.58$} & 130025 \linecomment{$\de = 1.56$} \\
    \hline
    $N_6(t) \leq$ & 5496 \linecomment{$\de = 2.08$} & 11760 \linecomment{$\de = 1.82$} & 21500 \linecomment{$\de = 1.7$} & 35345 \linecomment{$\de =1.63$} & 53843 \linecomment{$\de=1.59$} & 77601 \linecomment{$\de = 1.56$} & 107127 \linecomment{$\de = 1.54$} & 142570 \linecomment{$\de = 1.52$} & 185585 \linecomment{$\de = 1.5$} \\
    \hline
    $N_7(t) \leq$ & 6636 \linecomment{$\de = 1.93$} & 14904 \linecomment{$\de = 1.72$} & 27969 \linecomment{$\de = 1.63$} & 46672 \linecomment{$\de =1.57$} & 72070 \linecomment{$\de=1.51$} & 104117 \linecomment{$\de = 1.51$} & 144528 \linecomment{$\de = 1.49$} & 193850 \linecomment{$\de = 1.48$} & 252679 \linecomment{$\de = 1.47$} \\
    \hline
    $N_8(t) \leq$ & 7997 \linecomment{$\de = 1.83$} & 18611 \linecomment{$\de = 1.66$} & 35622 \linecomment{$\de = 1.58$} & 60084 \linecomment{$\de =1.53$} & 93038 \linecomment{$\de=1.5$} & 135954 \linecomment{$\de = 1.48$} & 189230 \linecomment{$\de = 1.46$} & 254173 \linecomment{$\de = 1.45$} & 331681 \linecomment{$\de = 1.44$} \\
    \hline
    $N_9(t) \leq$ & 9567 \linecomment{$\de = 1.76$} & 22870 \linecomment{$\de = 1.61$} & 44355 \linecomment{$\de = 1.54$} & 75590 \linecomment{$\de =1.5$} & 117868 \linecomment{$\de=1.47$} & 173110 \linecomment{$\de = 1.45$} & 241312 \linecomment{$\de = 1.44$} & 324969 \linecomment{$\de = 1.43$} & 425159 \linecomment{$\de = 1.42$} \\
    \hline
    $N_{10}(t) \leq$ & 11333 \linecomment{$\de = 1.7$} & 27682 \linecomment{$\de = 1.57$} & 54278 \linecomment{$\de = 1.51$} & 92866 \linecomment{$\de =1.475$} & 145904 \linecomment{$\de=1.45$} & 214051 \linecomment{$\de = 1.435$} & 300172 \linecomment{$\de = 1.425$} & 405179 \linecomment{$\de = 1.415$} & 531096 \linecomment{$\de = 1.41$} \\
    \hline
    \end{tabular}
    \caption{Numerics for \Cref{thm:finite-check}}
    \label{fig:finite-check-numerics}
\end{figure}

Therefore, in order to show \Cref{thm:finite-check}, it suffices to compute $D_k(r,t;n) - D_k(r+1,t;n)$ for $n \leq N_k(t)$ with the aid of a computer and determine all possible counterexamples which arise from these cases. All such counterexamples for $D_k(r,t;n) \geq D_k(r+1,t;n)$ satisfy $k=2,n \leq 8$, and similarly all such counterexamples for $D_k(r,t;n) > D_k(r+1,t;n)$ satisfy $n \leq 16$. This completes the proof.

\uspunctuation
\printbibliography

\end{document}